\documentclass[12pt,abstracton]{scrartcl}
\usepackage{amsmath,amsfonts,amscd,amssymb,amsthm}
\usepackage{graphicx}
\usepackage[english]{babel}
\usepackage[utf8]{inputenc}
\usepackage{mathrsfs}

\newtheorem{theorem}{Theorem}[section]
\newtheorem{definition}[theorem]{Definition}

\newtheorem{corollary}[theorem]{Corollary}
\newtheorem{lemma}[theorem]{Lemma}
\newtheorem{proposition}[theorem]{Proposition}

\newtheorem{remark}[theorem]{Remark}

\usepackage{hyperref}

\title{\Huge\sffamily\slshape Ruelle Operator for Continuous Potentials and DLR-Gibbs Measures}

\author{\large
  Leandro Cioletti
  \\[-0.3cm]
  \footnotesize Departamento de Matem\'atica - UnB
  \\[-0.3cm]
  \footnotesize 70910-900, Bras\'ilia, Brazil
  \\[-0.3cm]
  \footnotesize\texttt{cioletti@mat.unb.br}
  \and
  \large
  Artur O. Lopes
  \\[-0.3cm]
  \footnotesize Departamento de Matem\'atica - UFRGS
  \\[-0.3cm]
  \footnotesize 91509-900, Porto Alegre, Brazil
  \\[-0.3cm]
  \footnotesize\texttt{arturoscar.lopes@gmail.com}
  \and
  \large
  Manuel Stadlbauer
  \\[-0.3cm]
  \footnotesize Departamento de Matem\'atica - UFRJ
  \\[-0.3cm]
  \footnotesize 21941-909, Rio de Janeiro, Brazil
  \\[-0.3cm]
  \footnotesize\texttt{manuel@im.ufrj.br}
}

\date{\small\today}

\allowdisplaybreaks
\begin{document}
\makeatletter
\def\blfootnote{\gdef\@thefnmark{}\@footnotetext}
\let\@fnsymbol\@roman
\makeatother

\maketitle
\begin{abstract}
In this work we study the Ruelle Operator associated to
a continuous potential defined on a countable product of a compact  metric space.
We prove a generalization of Bowen's criterion
for the uniqueness of the eigenmeasures. One of the main results of the article is to show that a probability is DLR-Gibbs (associated to
a continuous translation invariant specification), if and only if, is an eigenprobability for the transpose of the Ruelle operator.

Bounded extensions of the Ruelle operator
to the Lebesgue space of integrable functions, with respect to
the eigenmeasures, are studied and the problem of existence of
maximal positive eigenfunctions for them is considered. One of our main
results in this direction is the existence of such positive eigenfunctions
for Bowen's potential in the setting of a compact and metric alphabet.
We also present a version of Dobrushin's Theorem in the setting of Thermodynamic Formalism.
\end{abstract}

\blfootnote{The authors are partially supported by CNPq-Brazil.}

\section{Introduction}\label{secao-introducao}

The classical Ruelle operator needs no introduction and nowadays
is a key concept of Thermodynamic Formalism. This operator was defined 1968 for infinite
dimension  by David Ruelle in
his seminal paper \cite{MR0234697} and since then has had attracted the attention
of the Dynamical Systems community. Remarkable applications of this
operator to Hyperbolic dynamical systems and Statistical Mechanics
were presented by Ruelle, Sinai and Bowen,
see \cite{MR2423393,MR0234697,MR0399421}, and in the presence of
a conformal measure, the action of the Ruelle operator and the transfer
operator from abstract ergodic theory coincide.
Nowadays transfer operators are present in several applications in pure
and applied mathematics and are fruitfully applied in many areas of active development,
see \cite{MR1793194} for comprehensive overview of the works
before 2000.

The the classical theory, Ruelle's operator is associated with the full shift $\sigma:\Omega \to \Omega$ on
the symbolic space $\Omega\equiv M^{\mathbb{N}}$, for $M=\{1,\ldots,n\}$, and acts on the space of H\"older
continuous functions.
%
In its classical form, the Ruelle operator $\mathscr{L}_{f}$ for a given continuous function $f:\Omega \to \mathbb{R}$
 is defined by $\mathscr{L}_{f}(\varphi)=\psi$,
where  $\psi$  is given by, for any $x\in\Omega$,
\[
\psi(x)
\,=\,
\sum_{y \in\Omega: \,\, \sigma(y)=x}
\!\!\!\!\!
e^{f(y)}\, \varphi(y).
\]
This operator is a useful tool for constructing equilibrium states, which are defined as those probability measures which maximize the variational problem
\begin{align}\label{variational-problem}
\sup_{\mu\in \mathscr{P}_{\sigma}(\Omega,\mathscr{F})}
\{h_{\mu}(\sigma) + \int_{\Omega} f\, d\mu\},
\end{align}
as proposed, e.g., by Ruelle in 1967 (\cite{MR0217610}) and Walters in 1975 (\cite{MR0390180}). In here, $h_{\mu}(\sigma)$ refers to the Kolmogorov-Sinai entropy of $\mu$
and $\mathscr{P}_{\sigma}(\Omega,\mathscr{F})$ to the
set of all $\sigma$-invariant Borel probability measures over $\Omega$. Even though the existence of the solution
to the variational problem can be easily obtained through abstract theorems
of convex analysis, the Ruelle operator approach
provides additional informations like uniqueness or decay of correlations
(see \cite{MR1793194} and \cite{MR1085356}). In the slightly more general situation of an expanding map on a compact and metric space, this approach was successfully implemented around the eighties and nineties for Hölder continuous potentials as well as for more general potentials in Walter's class  $W(\Omega,\sigma)$ or Bowen's class $B(\Omega,\sigma)$, see
\cite{Ruelle:1989,MR1841880,MR0466493,MR1783787,MR2342978}.

However, the setting in the references above require that the number of preimages of each point is finite, which excludes symbolic spaces like $\Omega=(\mathbb{S}^{n})^{\mathbb{N}}$ which are related
to several famous models of Statistical Mechanics.
For example, $n=1$ gives  us the so-called XY model, for $n=3$ we obtain the Heisenberg model and for $n=4$ the toy model for the Higgs sector of the Standard Model,
see \cite{MR595651,MR2864625,MR0269252,MR887102,GriffithsI-JMP,MR1552524,PhysRevLett.20.589}
for more details.

In \cite{MR2864625}  the authors
used the idea of an \textit{a priori} measure $p:\mathscr{B}(\mathbb{S}^1)\to [0,1]$
to circumvent the problem of uncountable alphabets
and developed the theory of the Ruelle operator for H\"older potentials
on $(\mathbb{S}^1)^{\mathbb{N}}$ with the dynamics
given by the left shift map.
This approach also works for H\"older potentials
when the unit circle $\mathbb{S}^1$ is replaced by a more
general compact metric space $M$, but one has to be careful about the choice of the
\textit{a priori} measure, see \cite{MR3377291} for details.

A similar setting is considered in \cite{MU}. The main
difference is that  the dynamics is more general, given by what the authors called a modified
shift, which is a map of the form $x\longmapsto (T(x_1),x_2,x_3,\ldots)$, where $T:M\to M$ is
only assumed to be an endomorphism.  Of course, if $T$ is the identity map we recover the usual
left shift. For this more general dynamical system a Thermodynamic Formalism is
constructed and a transfer operator is defined by using  an a priori measure as here.
Their main uniqueness result is based on the so-called Dini condition, while here we
adapt to the uncountable alphabet setting the classical Walters and Bowen's condition.

In this more general setting, the operator is defined as
\[\mathscr{L}_{f} (\varphi)\,(x) = \int_{M} e^{f(ax)}\, \varphi(ax)\, dp(a),\]
where $ax:=(a,x_1,x_2,\ldots)$.
A full support condition is imposed on the \textit{a priori} measure in \cite{MR3377291}
but this is not a strong restriction since in the majority of the applications, there is a natural
choice for this measure which always satisfies this condition.
For instance, for the classical Ruelle operator $\mathscr{L}_{f}$,
the metric space is the finite set $M=\{1,2,..,n\}$
and one usually considers $p$ is the  counting measure, and if $M$
is a general compact group, the one might consider the Haar measure.


Some of these results have a counterpart when $M$ is not compact. In case of countable alphabets, the applications
are motivated through applications to non-uniformly expanding or hyperbolic dynamical systems, see \cite{Sarig2015} and references therein. On the other hand, the $\phi^4$-model from Statistical Mechanics also motivates the study
of alphabets which are Polish spaces. In this setting, equilibrium states might only exist as positive linear functionals, but for summable Hölder potentials, the Ruelle operator still has a spectral gap (see \cite{CiolettiSilvaStadlbauer:2019}).


We point out that is very common in the literature in Statistical Mechanics that the main concepts are presented via translation invariant specifications (see \cite{MR2807681} and \cite{Cioletti20176139}). The so called DLR probabilities are defined in such way. We will present some new results in Thermodynamic Formalism under such point of view. We consider the case where the alphabet is a compact metric space.

In Statistical Mechanics as well as in Thermodynamic Formalism, existence and multiplicity
of DLR-Gibbs Measures plays an important role (see \cite{Sarig2015}), in particular due to
Dobrushin's interpretation of phase transitions. Even though there is no universal definition
of phase transition, they are nowadays understood as either the existence of more than one DLR-Gibbs measure, more than one
 equilibrium state or more than one eigenprobability for the dual of the Ruelle operator and so on
(see \cite{MR3350377,Cioletti20176139,MR2807681} for more details).

The goal of the present paper is to follow the Ruelle operator approach
for general continuous potentials defined over the infinite cartesian product of a general compact
metric space. A key observation in here is the characterization of extremal DLR-Gibbs measures
by their tail sigma-algebras (Theorem \ref{extreme vs exact}) and that
the set of DLR-Gibbs measures coincides with eigenprobabilities of the
dual of the Ruelle operator (Theorem \ref{mainDLR}). Moreover,
by extending the action of Ruelle's operator to the space of
integrable functions in Section \ref{secao-ruelle-L1},
Theorems \ref{extreme vs exact} and \ref{mainDLR} establish a further
characterization of extremal DLR-Gibbs measures through exactness and,
in particular, through Lin's characterization of exactness (see Corollary \ref{Lin}).

From the viewpoint of Dobrushin's interpretation of phase transitions,
Theorem \ref{mainDLR} allows to employ classical ideias from dynamical systems
in order to study phase transitions for
one-sided one-dimensional translation invariant statistical mechanics systems.
Namely, for potentials
satisfying a condition similar to the one by Bowen, we obtain in Theorem \ref{teo-unicidade-d1}
that there exists a unique DLR-Gibbs measure, and that this measure is exact.
Furthermore, we also develop a $C(\Omega)$-perturbation theory for the Ruelle operator and
present a constructive approach to solve the classical variational problem
for continuous potentials (see Section \ref{secao-convergencia-L1-operador-ruelle}).

Thereafter, we study in  Section \ref{secao-existencia-auto-funcao} the existence of integrable eigenfunctions
of the Ruelle operator for potentials beyond Bowen's class. Under
 a mild hypothesis on the potential and by approximating the potential
 uniformly by potentials $(f_n)$, we prove that $\limsup h_{f_n}$  is
 a non-trivial integrable eigenfunction associated to the maximal eigenvalue.
 An further relevant result in here is that $\limsup_{n\to\infty} \mathscr{L}^n_{f}(1)/\lambda_f^n$
 is an eigenfunction of $\mathscr{L}_{f}$ under fairly general conditions, which implies
 that a potential in the Bowen class admits an eigenfunction which is bounded
 from above and below (see Corollary \ref{cor-autofun-Bowen}).



\section{Preliminaries}\label{secao-preliminares}
Here and subsequently $(M,d)$ denotes a compact metric space endowed with
a Borel probability measure $p$ which is assumed to be fully supported on
$M$. Let $\Omega$ denote the infinite cartesian product $M^{\mathbb{N}}$
and $\mathscr{F}$ be the $\sigma$-algebra generated by its cylinder sets.
We will consider the dynamics on $\Omega$ given by
the left shift map $\sigma:\Omega\to\Omega$
which is defined, as usual, by $\sigma(x_1,x_2,\ldots)=(x_2,x_3,\ldots)$.
We use the notation $C(\Omega)$ for the space of all real continuous
functions on $\Omega$.
When convenient we call an element $f\in C(\Omega)$ a potential and
unless stated otherwise all the potentials are assumed to be  general
continuous functions.
The Ruelle operator associated to the potential $f$ is a mapping
$\mathscr{L}_f:C(\Omega)$  to $C(\Omega)$ that sends
$\varphi\mapsto \mathscr{L}_{f}(\varphi)$ which is defined for each
$x\in\Omega$ by
\begin{align}
\label{def-operador-de-ruelle}
\mathscr{L}_{f}(\varphi)(x)
=
\int_{M}\exp(f(ax))\varphi(ax)\, dp(a),
\ \ \text{where}\ \ ax:=(a,x_1,x_2,\ldots).
\end{align}
Due to compactness of $\Omega$  in the product topology
and the Riesz-Markov theorem we have that $C^*(\Omega)$
is isomorphic to $\mathscr{M}_s(\Omega,\mathscr{F})$,
the space of all signed Radon measures.
Therefore we can define $\mathscr{L}^*_{f}$, the dual of the Ruelle operator,
as the unique continuous map from $\mathscr{M}_s(\Omega,\mathscr{F})$
to itself satisfying for each $\gamma\in\mathscr{M}_s(\Omega,\mathscr{F})$
the following identity
\begin{align}\label{eq-dualidade}
\int_{\Omega} \mathscr{L}_{f}(\varphi)\, d\gamma
=
\int_{\Omega} \varphi\, d[\mathscr{L}^*_{f}\gamma]
\qquad \forall \varphi\in C(\Omega).
\end{align}
It follows from the positivity of $\mathscr{L}_f$ that
the map
$
\gamma
\mapsto
\mathscr{L}^*_{f}(\gamma)/
\mathscr{L}^*_{f}(\gamma)(1)
$
sends the space of all Borel probability measures
$\mathscr{P}(\Omega,\mathscr{F})$ to itself.
Since $\mathscr{P}(\Omega,\mathscr{F})$ is a convex and
compact in the weak topology (which is Hausdorff in this case)
and the mapping
$
\gamma
\mapsto
\mathscr{L}^*_{f}(\gamma)/
\mathscr{L}^*_{f}(\gamma)(1)
$
is continuous, the Schauder-Tychonoff theorem ensures the existence of
at least one Borel probability measure $\nu$ such that
$
\mathscr{L}^*_{f}(\nu) = \mathscr{L}^*_{f}(\nu)(1)\cdot \nu.
$
Notice that this eigenvalue
$
\lambda \equiv \mathscr{L}^*_{f}(\nu)(1)
$
is positive
but strictly speaking it could depend
on the choice of the fixed point
when it is not unique, however in any case such eigenvalues trivially
satisfies $\exp(-\|f\|_{\infty})\leq \lambda \leq \exp(\|f\|_{\infty})$
so we can always work with
\begin{align}
\label{definicao2-auto-valor-do-dual}
\lambda_{f}
=
\sup
  \left\{
	\mathscr{L}^*_{f}(\nu)(1):
	\begin{matrix}
	\nu\in \mathscr{P}(\Omega,\mathscr{F}) \ \text{and}\
	\nu\ \text{is a fixed point for}\
	\\
	\gamma
	\mapsto
	\mathscr{L}^*_{f}(\gamma)/
	\mathscr{L}^*_{f}(\gamma)(1)
	\end{matrix}
  \right\}.
\end{align}
Of course, from the compactness of $\mathscr{P}(\Omega,\mathscr{F})$
and continuity of $\mathscr{L}_{f}^*$, the supremum is attained
and therefore the set defined below is not empty.

\begin{definition}[$\mathscr{G}^{*}(f)$] Let $f$ be a continuous potential
and $\lambda_f$ given by \eqref{definicao2-auto-valor-do-dual}.
We define
\[
\mathscr{G}^{*}(f)
=
\{
	\nu\in\mathscr{P}(\Omega,\mathscr{F}):
	\mathscr{L}^*_f\nu=\lambda_f\nu
\}.
\]

\end{definition}

To study the eigenfunctions of $\mathscr{L}_{f}$,
where $f$ is a general continuous potential,
we will need the following version of the Ruelle-Perron-Frobenius theorem for H\"older potentials (see \cite{MR2864625} and \cite{MR3377291} for the proof).

We consider the metric $d_{\Omega}$ on $\Omega$ given by
$
d_{\Omega} (x,y)
=
\sum_{n=1}^{\infty} 2^{-n}d(x_n,y_n)
$
and for any fixed $0< \gamma\leq 1$ we denote by $C^{\gamma}(\Omega)$
the space of all $\gamma$-H\"older continuous functions,
i.e, the set of all functions $\varphi:\Omega\to\mathbb{R}$ satisfying
\[
\mathrm{Hol}_{\gamma}(\varphi)
=
\sup_{x,y\in\Omega: x\neq y}
\dfrac{|\varphi(x)-\varphi(y)|}{d_{\Omega}(x,y)^{\gamma}}
<+\infty.
\]

\begin{theorem}[Ruelle-Perron-Frobenius]\label{teo-RPF-holder}
Let $(M,d)$ be a compact metric space, $\mu$ a
Borel probability  measure of full support on $M$ and
$f$ be a potential in $C^{\gamma}(\Omega)$,
where $0<\gamma<1$. Then
$\mathscr{L}_f: C^{\gamma}(\Omega) \to C^{\gamma}(\Omega)$
has a simple positive eigenvalue of maximal modulus
$\lambda_f$ and there is a strictly positive function $h_f$
satisfying $\mathscr{L}_{f}(h_f)=\lambda_{f}h_{f}$
and a Borel probability measure $\nu_{f}$ for which
$\mathscr{L}^{*}_f(\nu_f)= \lambda_{f}\nu_{f}$ and
$\mathscr{L}^{*}_f(\nu_f)(1)= \lambda_{f}$.
\end{theorem}

\section{The Pressure of Continuous Potentials}
\label{secao-pressao}

The next proposition is an extension of Corollary 1.3 in \cite{MR1143171}.
Here $M$ is allowed to be any general compact metric space. It is worth mentioning that Sarig in \cite{Sarig2015} obtained a similar result for the Gurevich pressure
with respect to countable alphabets, but the techniques employed in our proof are much simpler.

\begin{proposition}\label{prop-limite-pressao}
Let $f\in C(\Omega)$ be a potential and
$\lambda_{f}$ given by \eqref{definicao2-auto-valor-do-dual}.
Then, for any $x\in\Omega$ we have
\[
\lim_{n\to\infty}\frac{1}{n}\log \mathscr{L}^n_{f}(1)(\sigma^n x)
=
\log\lambda_f.
\]
\end{proposition}
\begin{proof}

Let $\nu\in \mathscr{G}^{*}(f)$ a fixed eigenprobability.
Without loss of generality we can assume that $\mathrm{diam}(M)=1$.
By the definition of $d_{\Omega}$ for any pair $z,w\in\Omega$
such that $z_i=w_i,\ \forall i=1,\ldots,N$ we have that
$d_{\Omega}(z,w)\leq 2^{-N}$. From uniform continuity of $f$
given $\varepsilon>0$, there is $N_0\in\mathbb{N}$, so that
$|f(z)-f(w)|<\varepsilon/2$, whenever $d_{\Omega}(z,w)<2^{-N_0}$.
If $n>2N_0$ and $a:=(a_1,\ldots,a_n)$ we claim that
\begin{align}\label{des-Snf}
\|S_n(f)(ax)-S_n(f)(ay)\|_{\infty}
\leq (n-N_0)\frac{\varepsilon}{2}+2\|f\|_{\infty}N_0,
\end{align}
where $S_n(f) \equiv f+f\circ \sigma +\ldots+ f\circ\sigma^{n-1}$.
Indeed, for any $n\geq 2N_0$ we have
\begin{align*}
|S_n(f)(ax)-S_n(f)(ay)|
&=
|
\sum_{j=0}^{n-1}  f (\sigma^j (a_1,\ldots,a_n,x)
-
\sum_{j=0}^{n-1}  f (\sigma^j (a_1,\ldots,a_n,y)
|
\\
&\hspace*{-1cm}\leq
\sum_{j=0}^{n-1-N_0}
|f (\sigma^j (a_1,\ldots,a_n,x) -  f (\sigma^j (a_1,\ldots,a_n,y) |
+
\\
&
\sum_{j=0}^{N_0}
|f (\sigma^j (a_{n-N_0},\ldots,a_n,x) -  f (\sigma^j (a_{n-N_0},\ldots,a_n,y) |
\\
&\hspace*{-1cm}\leq
(n-N_0) \frac{\varepsilon}{2} + 2N_0\|f\|_{\infty}.
\end{align*}
The last inequality comes from the uniform continuity for the first terms and
from the uniform norm of $f$ for the second ones.

We recall that for any probability space $(E,\mathscr{E},\mathbb{P})$,
$\varphi$ and $\psi$ bounded real $\mathscr{E}$-measurable
functions the following inequality holds
\begin{align}\label{des-log-integral}
\left|
\log  \int_{E} e^{\varphi(\omega)} d\mathbb{P}(\omega)
-
\log \int_{E} e^{\psi(\omega)} d\mathbb{P}(\omega)
\right|
\leq
\|\varphi-\psi\|_{\infty}.
\end{align}
From the definition of the Ruelle operator, for any $n\in\mathbb{N}$,
we have
\[
\mathscr{L}^n_{f}(1)(\sigma^n x)
=
\int_{M^n} \exp(S_n(f)(a\sigma^n x))\, \prod_{i=1}^{n}dp(a_i)
\]
and from \eqref{des-Snf} and
\eqref{des-log-integral} with
$\varphi(a)=S_n(f)(a\sigma^n x)$ and
$\psi(a)=S_n(f)(ay)$
we get for $n\geq \max\{2N_0,4\varepsilon^{-1}\|f\|_{\infty}N_0\}$
the  estimate
\[
\frac{1}{n}
|\log(\mathscr{L}^n_{f}(1)(\sigma^n x))
-\log(\mathscr{L}^n_{f}(1)(y))|
\leq
\frac{1}{n}((n-N_0)\frac{\varepsilon}{2}+\frac{2\|f\|_{\infty}N_0}{n}
\leq
\varepsilon.
\]
By using Fubini's theorem, sum and subtract $\exp(S_n(f)(a\sigma^{n}y))$,
the identity \eqref{eq-dualidade}
iteratively and the last inequality
for $n\geq \max\{2N_0,4\varepsilon^{-1}\|f\|_{\infty}N_0\}$
we obtain
\begin{align*}
&\mathscr{L}^n_{f}(1)(\sigma^n x)
=
\int_{M^n} \exp(S_n(f)(a\sigma^n x))\, \prod_{i=1}^{n}dp(a_i)
\\
&=
\int_{M^n}\int_{\Omega}
\exp(S_n(f)(a\sigma^n x))\, d\nu (y)\prod_{i=1}^{n}dp(a_i)
\\
&\leq
\exp((n-N_0)\frac{\varepsilon}{2}+2\|f\|_{\infty}N_0)
\int_{M^n}\int_{\Omega}
\exp(S_n(f)(ay))
\, d\nu_f(y)\prod_{i=1}^{n}dp(a_i)
\\
&\leq
\exp(n\varepsilon)
\int_{\Omega}
\mathscr{L}^{n}_{f}(1)( y)
\, d\nu (y)
\\
&=
\exp(n\varepsilon)
\lambda_f^n.
\end{align*}
Similarly we obtain the lower bound
$\mathscr{L}^n_{f}(1)(\sigma^n x)\geq \exp(-n\varepsilon)\lambda_f^n$
so the proposition follows.
\end{proof}

\begin{corollary} \label{cor:same lambda for all}
Let $f$ be a continuous potential. If $\nu$ and $\hat{\nu}$
are fixed points for the map $\gamma\mapsto \mathscr{L}^{*}_{f}(\gamma)/\mathscr{L}^{*}_{f}(\gamma)(1)$
then $\mathscr{L}^{*}_{f}(\nu)(1)=\mathscr{L}^{*}_{f}(\hat{\nu})(1)=\lambda_{f}$.
\end{corollary}
\begin{proof}
For any $x_0\in\Omega$, by repeating the same steps of the proof of the
previous proposition one shows that
$
\log(\mathscr{L}^{*}_{f}(\nu)(1))
\equiv
\log(\lambda_f(\nu))
=
\lim_{n\to\infty}\frac{1}{n}\log \mathscr{L}^n_{f}(1)(x_0)
=
\log(\lambda_f(\hat{\nu}))
=
\log(\mathscr{L}^{*}_{f}(\hat{\nu})(1)).
$
\end{proof}

\begin{definition}[Pressure Functional]
The function $p:C(\Omega)\to\mathbb{R}$ given by
$p(f)=\log\lambda_f $ is called pressure functional.
\end{definition}

In classical Thermodynamic Formalism, where $M$ is a finite alphabet,
the pressure functional usually refers to
the function $P:C(\Omega)\to \mathbb{R}$ given by
\[
P(f)\equiv
\sup_{\mu\in \mathscr{P}_{\sigma}(\Omega,\mathscr{F})}
\{h_{\mu}(\sigma) + \int_{\Omega} f\, d\mu\}.
\]
After developing some perturbation theory we will
show latter that both definitions of the Pressure
functional are equivalent for any continuous
potential, i.e., $P=p$.

Since $\Omega$ is compact and the space of all
$\gamma$-H\"older continuous functions $C^{\gamma}(\Omega)$
is an algebra that separate points and contains
the constant functions, we can apply the Stone-Weierstrass
theorem to conclude that the closure of $C^{\gamma}(\Omega)$
in the uniform topology is $C(\Omega)$.
Therefore, for any arbitrary continuous potential $f$,
there is a sequence $(f_n)_{n\in\mathbb{N}}$ of H\"older continuous
potentials such that
$\|f_n-f\|_{\infty}\to 0$, when $n\to\infty$.
For such uniform convergent sequences we will see that
$p(f_n)$ converges to $p(f)$.  Actually, in our case the sequence $(f_n)_{n\in\mathbb{N}}$
can be explicitly exhibited, see Section \ref{secao-aplicacoes}. Moreover, a much
stronger result holds. That is, the pressure functional
is Lipschitz continuous function on $C(\Omega)$.

\begin{proposition}\label{prop-pressao-lip}
If $f,g:\Omega\to\mathbb{R}$ are two arbitrary continuous
potentials then the following inequality holds $|p(f)-p(g)|\leq \|f-g\|_{\infty}$.
\end{proposition}
\begin{proof}
The proof is an immediate consequence of Proposition \ref{prop-limite-pressao}
and inequality \eqref{des-log-integral}.
\end{proof}

\begin{corollary}\label{cor-converg-lambdan}
Let $(f_n)_{n\in\mathbb{N}}$ be a sequence of continuous potentials
such that $f_n\to f$ uniformly, then $p(f_n)\to p(f)$.
In particular, $\lambda_{f_n}\to\lambda_{f}$.
\end{corollary}

\section{DLR-Gibbs Measures and Eigenmeasures}
\label{secao-equiv-dlr-automedidas}

In this section we discuss the concept of specifications in the setting of
Thermodynamic Formalism. Some of its elementary properties
for finite state spaces are discussed in details within this framework
in  \cite{Cioletti20176139}.

For each $n\in \mathbb{N}$, we define the projection on the $n$-th
coordinate
$\pi_n:\Omega\to M$ by $\pi_n(x)=x_n$. We use the notation
$\mathscr{F}_n$ to denote the sigma-algebra generated by the
projections $\pi_1,\ldots,\pi_n$.
On the other hand, the notation $\mathscr{T}_n$ stands for the sigma-algebra
generated by the collection of projections $\{\pi_{k}: k\geq n+1\}$.
The so-called tail sigma-algebra (or terminal sigma-algebra) is defined as $\mathscr{T}\equiv \cap_{n\in\mathbb{N}} \mathscr{T}_n$.
A cylinder set in $\Omega$ is a set of the form $\pi^{-1}_{1}(E_1)\cap \ldots \cap\pi^{-1}_{k}(E_k)$,
where $E_1,\ldots,E_k\in \mathscr{B}(M)$, the Borel sigma-algebra of $M$.

\bigskip

Let $f\in C(\Omega)$ a potential and for each $n\in\mathbb{N}$,
$x\in\Omega$ and $E\in\mathscr{F}$ consider the mapping
$K_n:\mathscr{F}\times\Omega\to [0,1]$ given by
\begin{align}\label{def-Kn-ruelle}
K_n(E,x)\equiv
\frac{\mathscr{L}_{f}^{n}(1_E)(\sigma^n(x)) }
{\mathscr{L}^{n}_{f}(1)(\sigma^n(x))}.
\end{align}
For any fixed $x\in\Omega$, the monotone convergence theorem implies that the
map $\mathscr{F}\ni E\mapsto K_n(E,x)$ is a
probability measure. For any fixed measurable
set $E\in\mathscr{F}$ follows from the Fubini theorem that
the map $x\mapsto K_n(E,x)$ is $\mathscr{T}_{n}$-measurable.
So $K_n$ is a \textbf{probability kernel} from $\mathscr{T}_{n}$
to $\mathscr{F}$.

Notice that, for any $\varphi\in C(\Omega)$, the kernel $K_n(\varphi,x)$ is
well defined due to the right hand side of \eqref{def-Kn-ruelle}.
It is easy to see (using the rhs of \eqref{def-Kn-ruelle})
that they are proper kernels, meaning that
for any bounded $\mathscr{T}_{n}$-measurable function
$\varphi$, we have $K_n(\varphi,x)=\varphi(\sigma^n(x))$.
The above probability kernels have the following important property.
For any fixed continuous function $\varphi$, the map
$x\mapsto K_n(\varphi,x)$ is continuous as a consequence of the Lebesgue dominated
convergence theorem. We refer to this saying that $(K_n)_{n\in\mathbb{N}}$
has the Feller property.

\medskip

\begin{definition} \label{esp} A Gibbsian specification with parameter set $\mathbb{N}$
in the translation invariant setting is an abstract family of
probability Kernels $K_n:(\mathscr{F},\Omega)\to [0,1]$,
$n\in\mathbb{N}$ such that

\begin{enumerate}
		\item[a)]
		$\Omega\ni x\to K_{n}(E,x)$ is $\sigma^{n}\mathscr{F}$-measurable
		function for any $E\in\mathscr{F}$;
		
		\item[b)] $\mathscr{F}\ni E\mapsto K_{n}(E,x)$ is a probability measure
		for any $x\in\Omega$;
		
		\item[c)] for any $n,r\in\mathbb{N}$ and any bounded $\mathscr{F}$-measurable function
		$f:\Omega\to\mathbb{R}$ we have \textbf{the compatibility condition}, i.e.,
		\begin{align*}
		K_{n+r}(f,x)
		=
		\int_{\Omega} K_{n}(f,\cdot) dK_{n+r}(\cdot,x)
		\equiv
		K_{n+r}(K_n(f,\cdot),x).
		\end{align*}
\end{enumerate}
\end{definition}

\begin{remark}
The classical definition of specification as given in \cite{MR2807681}
requires even in our setting a larger family of probability kernels.
To be more precise we
have to define a probability kernel for any finite subset
 $\Lambda\subset\mathbb{N}$
and the kernels $K_{\Lambda}$ have to satisfy a), b) and a generalization of c).
In the translation invariant setting (in the sense that the one site influence has same expression for every site) on the lattice $\mathbb{N}$, the formalism can be simplified and one
needs only to consider the family $K_n$, $n\in \mathbb{N},$
as defined above.

Strictly speaking, in order to use the results in \cite{MR2807681},
one first has to extend our specifications to
any set $\Lambda=\{n_1,\ldots,n_r\}$, but this can
be consistently done by putting $K_{\Lambda}\equiv K_{n_r}$.
This simplified definition adopted here is further justified by the
fact that the DLR-Gibbs measures, compatible with a specification
with parameter set $\mathbb{N}$,
are completely determined by the kernels indexed in
any cofinal collection of
subsets of $\mathbb{N}$. So here we are taking advantage of
this result to define our kernels only on the cofinal
collection of subsets of $\mathbb{N}$ of the form
$\{1,\ldots,n\}$ with $n\in\mathbb{N}$. Therefore, when writing
$K_n$, we are really thinking, in terms of
the general definition of specifications,
of $K_{\{1,\ldots,n\}}$.
\end{remark}

\medskip
The only specifications needed here are
the ones given in (\ref{def-Kn-ruelle}),
which are defined in terms of any continuous potential $f$.
Notice that in the translation invariant setting,
the construction in (\ref{def-Kn-ruelle})
for the lattice $\mathbb{N}$ extends the usual construction
made in terms of regular interactions. But in any case, the kernels in
(\ref{def-Kn-ruelle}) give us particular constructions of
quasilocal specifications which allow us to use some of the
results from \cite{MR2807681}. We refer the reader to
\cite{Cioletti20176139} and \cite{MR1783787} for results about specifications
 when the dynamics have the finite pre-images property.

\medskip

Using $\mathscr{L}^n_f(\varphi\cdot \psi \circ \sigma^n) = \mathscr{L}^n_f(\varphi) \psi$ one easily obtains the following identity   for any  $r,n\in\mathbb{N}$, $x\in\Omega$ and $\varphi\in C(\Omega)$ (see, e.g., the proof of Theorem 4.3 in \cite{Cioletti20176139}).
\begin{align}\label{rel-consist-operador-ruelle}
  \mathscr{L}^{n+r}_f (\varphi) (\sigma^{n+r} (x))
  =
  \mathscr{L}^{n+r}_f
  \Big(
  \frac{\mathscr{L}^{n}_f\,\,(\varphi) (\,\sigma^n (\cdot)\,)}
  {\mathscr{L}^{n}_f(1) (\sigma^n (\cdot)\,)}
  \Big)
  (\sigma^{n+r} (x)).
\end{align}
The above identity immediately implies for the Kernels
defined by \eqref{def-Kn-ruelle} that
\begin{align}\label{eq-dlr}
K_{n+r}(f,x)
=
\int_{\Omega} K_{n}(f,\cdot) dK_{n+r}(\cdot,x)
\equiv
K_{n+r}(K_n(f,\cdot),x).
\end{align}
We refer to the above set of identities
as compatibility conditions for the family of probability kernels
$(K_n)_{n\in\mathbb{N}}$ or simply DLR equations.
Similar kernels are also considered in \cite{MR1783787}, but here
we are working with a dynamical system that may have uncountable many elements
in the preimage of any point.

\begin{definition}
We say that $\mu\in \mathscr{P}(\Omega,\mathscr{F})$
is a DLR-Gibbs measure for the continuous
potential $f$ if for any $n\in \mathbb{N}$ and any continuous function
$\varphi:\Omega \to \mathbb{R}$ we have for $\mu$-almost all $x$
that
\[
\mathbb{E}_\mu[\varphi|\mathscr{T}_{n}](x)=
\int_{\Omega} \varphi(y) \,d K_n(y,x).
\]
The set of all DLR-Gibbs measures for $f$ is
denoted by $\mathscr{G}^{DLR}(f)$.

\end{definition}

One very important and elementary result on DLR-Gibbs measure
is the equivalence between the two conditions below:
\begin{itemize}
	\item[a)] $\mu\in\mathscr{G}^{DLR}(f)$;
	
	\item[b)] for any $n\in \mathbb{N}$
				  and $E\in\mathscr{F}$ we have that
				  $\mu(E)=\int_{\Omega} K_{n}(E,\cdot)\, d\mu$.
\end{itemize}

We now prove that $\mu\in\mathscr{G}^{DLR}(f)$ is not
empty. The result of the next lemma for countable $M$ can also be found in \cite{Sarig2015}. For these state spaces it is also possible to allow some less regularity than continuit.

\begin{lemma}\label{lemma-star-subset-DLR}
Let $f\in C(\Omega)$ and $(K_n)_{n\in\mathbb{N}}$
as in \eqref{def-Kn-ruelle}. Then
$ \mathscr{G}^{*}(f)\subset \mathscr{G}^{DLR}(f)$.
\end{lemma}

\begin{proof}
Let $\nu$ be such that $\mathscr{L}^*_f\nu=\lambda_f\nu$ and
$\varphi$ a bounded $\mathscr{F}$-measurable function.
Notice that the quotient appearing in the first
integral below is $\mathscr{T}_{n}$-measurable.
Therefore for any bounded $\mathscr{F}$-measurable $\psi$
the following equality holds.
\begin{align*}
\int_{\Omega} (\varphi \circ \sigma^n) (x)\,
\frac{\mathscr{L}_{ f}^n \, (\psi )\, (\sigma^n (x))}
{ \mathscr{L}_{\,f}^n \, (1 )\, (\sigma^n (x)) }\, d\,\nu (x)
&=
\int_{\Omega}\,
\frac{\mathscr{L}_{f}^n \, (\psi \,  (\varphi \circ \sigma^n) )\, (\sigma^n (x))}
{ \mathscr{L}_{\,f}^n \, (1 )\, (\sigma^n (x)) }\, d\,\nu (x)
\\[0.4cm]
&\hspace*{-1.5cm}=
\int_{\Omega}\, \frac{1}{\lambda^n}\,\mathscr{L}_{ f}^n
	\left[  \frac{\mathscr{L}_{ f}^n \, (\psi \,  (\varphi \circ \sigma^n) )\, (\sigma^n (\cdot))}
				{ \mathscr{L}_{\,f}^n \, (1 )\, (\sigma^n (\cdot)) }
	\right](x)\, d\,\nu (x).
\end{align*}
By using the equation \eqref{rel-consist-operador-ruelle}
we see that the rhs above is equal to
\begin{align*}
\int_{\Omega}\,
\frac{1}{\lambda^n}\,\mathscr{L}_{ f}^n
		(\psi \,  (\varphi \circ \sigma^n) )(x)
		\, d\,\nu (x)
&=
\int_{\Omega}  \psi(x)\, (\varphi \circ \sigma^n)(x)\,\, d\,\nu (x).
\end{align*}
Since $\varphi$ is an arbitrary $\mathscr{F}$-measurable
function we can conclude that
$$
	\nu[E|\mathscr{T}_{n}](y)
	=
	\frac{\mathscr{L}_{f}^n \, (I_E )\, (\sigma^n (y))}
	{ \mathscr{L}_{f}^n \, (1 )\, (\sigma^n (y)) }
	\quad
	\nu-\text{a.s.}
$$
so the equation (\ref{def-Kn-ruelle}) implies that $\nu\in\mathscr{G}^{DLR}(f)$.
\end{proof}

\begin{corollary}
Let $f\in C(\Omega)$ be a potential and $(K_n)_{n\in\mathbb{N}}$
the specification defined by \eqref{def-Kn-ruelle}.
If $\mathscr{G}^{DLR}(f)$ is a singleton, then
$ \mathscr{G}^{*}(f) = \mathscr{G}^{DLR}(f)$.
\end{corollary}

The next lemma  establishes the reverse inclusion between
$\mathscr{G}^{DLR}(f)$ and   $\mathscr{G}^{*}(f)$.
Before, we state an interesting result from specification theory
which essentially says that the tail sigma-algebra $\mathscr{T}$
has to be trivial with respect to an extremal DLR-Gibbs measure.

%
%
%
%
\begin{theorem}\label{extreme vs exact}
Let $(K_n)_{n\in\mathbb{N}}$ be the specification given
in \eqref{def-Kn-ruelle}. Then the following conclusion
holds. A probability measure $\mu\in \mathscr{G}^{DLR}(f)$ is extreme
in $\mathscr{G}^{DLR}(f)$ if and only if $\mu$ is trivial
on $\mathscr{T}$. Consequently,
if $\mu$ is extreme in $\mathscr{G}^{DLR}(f)$, then every
$\mathscr{T}$-measurable function
$f$ is constant $\mu$ a.s..
\end{theorem}
The proof of the theorem adapted to our setting can be found in the appendix below.

\begin{lemma} \label{cheese}
Let $f\in C(\Omega)$ be a potential and $(K_n)_{n\in\mathbb{N}}$
defined as in \eqref{def-Kn-ruelle}. If $\mu$ is an extreme element in $\mathscr{G}^{DLR}(f)$
then $\mu\in \mathscr{G}^{*}(f)$.
\end{lemma}

%

\begin{proof}
This proof is based on the following algebraic identity
\begin{equation}\label{eq-aux-Lemma-DLR-automedida}
K_n(\mathscr{L}_{f}(\varphi),\sigma(y))
=
K_n(\mathscr{L}_{f}(1),\sigma(y))\ K_{n+1}(\varphi,y)
\end{equation}
which holds for every $n\in \mathbb{N}$, $y\in \Omega$ and $\varphi\in C(\Omega)$.
Indeed,
\begin{align*}
K_n(\mathscr{L}_{f}(\varphi),\sigma(y))
=
\frac{\mathscr{L}^{n+1}_{f}(\varphi)(\sigma^{n+1}(y))}{\mathscr{L}^{n}_{f}(1)(\sigma^{n+1}(y))}
&=
\frac{\mathscr{L}^{n+1}_{f}(1)(\sigma^{n+1}(y))}{\mathscr{L}^{n}_{f}(1)(\sigma^{n+1}(y))} \frac{\mathscr{L}^{n+1}_{f}(\varphi)(\sigma^{n+1}(y))}{\mathscr{L}^{n+1}_{f}(1)(\sigma^{n+1}(y))}
\\
&=
K_n(\mathscr{L}_{f}(1),\sigma(y))\ K_{n+1}(\varphi,y).
\end{align*}
Let $\mu \in\mathscr{G}^{DLR}(f)$.
By definition of a DLR-Gibbs measure and the convergence of backward martingales,
we have
$\mathbb{E}_\mu[\varphi|\mathscr{T}_{n}] = K_n(\varphi,\,\cdot\,)$
converges almost surely for any $\varphi \in C(\Omega)$.
It follows from separability of $C(\Omega)$ that there exists a set
of full measure such that
$\int_{\Omega} \varphi\,  d\mu_y := \lim_{n\to\infty} K_n(\varphi,y)$ for all $\varphi \in C(\Omega)$
and $y$ in this set. By completing the sigma-algebra, we moreover may assume that the set $\Omega'$ defined as the set of $y$ such that $\lim_{n\to\infty} K_n(\varphi,y)$ for all $\varphi \in C(\Omega)$ is measurable. We now show that $\sigma^{-1}(\Omega') = \Omega'$. In order to do so, assume that $y \in \sigma^{-1}(\Omega')$. It then follows from \eqref{eq-aux-Lemma-DLR-automedida} that
\begin{align}\label{eq-aux1-quociente-Lnf}
\int_{\Omega} \mathscr{L}_{f}(\varphi) d\mu_{\sigma y}
&=
\lim_{n\to\infty} K_n(\mathscr{L}_{f}(\varphi),\sigma(y))
=
\lim_{n \to \infty} K_n(\mathscr{L}_{f}(1),\sigma(y))\ K_{n+1}(\varphi,y)
\nonumber\\[0.3cm]
&=
{\int_{\Omega} \mathscr{L}_{f}(1) d\mu_{\sigma y}}
\ \lim_{n \to \infty}
 K_{n+1}(\varphi,y)
=
\lambda_{\sigma y}\ \lim_{n \to \infty}    K_n(\varphi,y),
\end{align}
where $\lambda_{\sigma y} \equiv \int_{\Omega} \mathscr{L}_{f}(1)\, d\mu_{\sigma y}$
exists since $\sigma y \in \Omega'$. Hence,  $\lim_{n \to \infty}  K_n(\varphi,y)$
exists and therefore, 
 $\sigma^{-1}(\Omega') \subset \Omega'$. 
 Now assume that $y \in \Omega'$. By substituting $\varphi$ with $\varphi \circ \sigma / \mathscr{L}_{f}(1) \circ \sigma$, the same argument implies that  
 \begin{align*}\label{eq-aux1-quociente-Lnf}
 \int \frac{\varphi \circ \sigma}{\mathscr{L}_{f}(1) \circ \sigma} d\mu_{y}
 &
 = \lim_{n \to \infty} \frac{K_n(\mathscr{L}_{f}(\varphi \circ \sigma / \mathscr{L}_{f}(1) \circ \sigma),\sigma(y))}{K_n(\mathscr{L}_{f}(1),\sigma(y))}\\ & = \lim_{n \to \infty} \frac{K_n(\varphi),\sigma(y))}{K_n(\mathscr{L}_{f}(1),\sigma(y))}. 
 \end{align*}
Setting $\varphi=1$, we obtain that $\lim_{n} K_n(\mathscr{L}_{f}(1),\sigma(y))$ exists. Hence, by repeating the argument for general $\varphi$, it follows that also $\lim_{n} K_n(\varphi,\sigma(y))$ exists. Hence, $\sigma(y) \in \Omega'$
and consequently $\sigma^{-1}(\Omega') = \Omega'$. 

Hence, we have by \eqref{eq-aux1-quociente-Lnf}, that 
\[
 \int_{\Omega} \varphi d\mu_{y} = \frac{1}{\lambda_{\sigma y}} \int_{\Omega} \mathscr{L}_{f}(\varphi) d\mu_{\sigma y}  
\]
for $\mu$-almost $y$. In particular,  $\int_{\Omega} \varphi d\mu_{y} =  \int_{\Omega} \varphi d\mu_{\tilde{y}} $ whenever $\sigma(y) = \sigma(\tilde{y})$. In particular, the map $y \mapsto \int_{\Omega} \varphi d\mu_{y}$ is trivial with respect to $\mathscr{T}$ and therefore, as $\mu$ is extremal, constant. Furthermore, assume that $\lambda = \lambda_y$ a.s.. Then, 
by applying bounded convergence three times,
\begin{align*}
 \frac{1}{\lambda}\int_{\Omega} \mathscr{L}_{f}(\varphi) d\mu & = \lim_{n \to \infty}  \frac{1}{\lambda} \int_{\Omega} K_n(\mathscr{L}_{f}(\varphi))(y)  d\mu
=   \int_{\Omega} \lim_{n \to \infty} \frac{K_n(\mathscr{L}_{f}(\varphi))(y)}{K_n(\mathscr{L}_{f}(1))(y)}  d\mu \\
& = \int \lim_{n \to \infty} K_n(\varphi)(y) d\mu = \int \varphi d\mu.
\end{align*}
From Corollary \ref{cor:same lambda for all} it is a simple matter to check that $\lambda = \int_{\Omega} \mathscr{L}_{f}(1)\, d\mu$ is the spectral
radius of $\mathscr{L}_{f}$ acting on $C(\Omega)$ and therefore
$\mathscr{L}_{f}^\ast (\mu) = \lambda_f \mu$.
That is, $\mu\in \mathscr{G}^{*}(f)$.
\end{proof}

Observe that the combination of Theorem \ref{extreme vs exact}
with Lemma \ref{cheese} identifies the extreme elements in
$\mu \in\mathscr{G}^{DLR}(f)$ with the set of exact, conformal measures,
that is those elements in $ \mathscr{G}^{*}(f)$
for which $\mathscr{T}$ is trivial. In particular, after
extending the action of $\mathscr{L}_{f}$ to $L^1(\mu)$,
Lin's criterion provides a further characterization of these measure
as given in Corollary \ref{Lin} below.
As an immediate corollary of Lemma \ref{cheese}, we obtain the main
result of this section.

\begin{theorem} \label{mainDLR} Let $f\in C(\Omega)$ and $(K_n)_{n\in\mathbb{N}}$
as in \eqref{def-Kn-ruelle}. Then
$\mathscr{G}^{DLR}(f)= \mathscr{G}^{*}(f)$.
\end{theorem}

\begin{proof} By Lemma  \ref{lemma-star-subset-DLR},  for any continuous potential $f$,
we have  $\mathscr{G}^{*}(f)  \subset \mathscr{G}^{DLR}(f)$.
 On the other hand, Lemma \ref{cheese} ensures $\mathrm{ex}(\mathscr{G}^{DLR}(f))\subset \mathscr{G}^{*}(f)$. By
compactness and the Krein-Milman theorem it  follows that
$\mathscr{G}^{DLR}(f)\subset \mathscr{G}^{*}(f)$, thus proving the theorem.
\end{proof}

\section{Uniqueness Theorem for Eigenprobabilities}
\label{secao-teo-unicidade}
\begin{theorem}\label{teo-unicidade}
Let $f$ be a continuous potential and
$(K_n)_{n\in\mathbb{N}}$ be the specification defined as
in \eqref{def-Kn-ruelle}. Suppose that there is a constant $c>0$
such that for every cylinder set $F\in\mathscr{F}$
there is $n\in\mathbb{N}$ such that
\[
K_n(F,x)\geq c K_n(F,y)
\]
for all $x,y\in\Omega$.
Then, the set
$
\mathscr{G}^{*}(f)
$
has only one element.
\end{theorem}

\begin{proof}
Because of Lemma \ref{lemma-star-subset-DLR} it is enough to show that $\mathscr{G}^{DLR}(f)$
is a singleton.
Suppose that $\mathscr{G}^{DLR}(f)$ contains
two distinct elements $\mu$ and
$\nu$. Then the convex combination
$(1/2)(\mu+\nu)\in \mathscr{G}^{DLR}(f)\setminus \mathrm{ex}(\mathscr{G}^{DLR}(f))$,
where $\mathrm{ex}(\mathscr{G}^{DLR}(f))$ denotes the set of extreme measures
of $\mathscr{G}^{DLR}(f)$. Therefore it is sufficient to show that
$\mathscr{G}^{DLR}(f)\subset \mathrm{ex}(\mathscr{G}^{DLR}(f))$
.

Let $\mu \in \mathscr{G}^{DLR}(f)$, $E_0\in \mathscr{T}$
and suppose that $\mu(E_0)>0$. The existence of such set is ensured by the
Theorem 7.7 item (c) in \cite{MR2807681}, which says that any element
$\mu\in\ \mathscr{G}^{DLR}(f)$ is uniquely determined by its restriction to the
tail sigma-algebra $\mathscr{T}$ (see Corollary \ref{ap5} in the Appendix).
Since $\mu(E_0)>0$ the probability measure
$\nu\equiv \mu(\cdot|E_0)\in \mathscr{G}^{DLR}(f)$,
see Theorem 7.7 (b) in \cite{MR2807681} (or, see Proposition \ref{ap2} and Corollary \ref{ap4} in the  Appendix).

We now prove that for all $E\in\mathscr{F}$, we have
$\nu(E) \geq c\mu(E)$ for some $c>0$. Fix a cylinder set $F\in\mathscr{F}$.  Then,
for $n$ sufficiently large, it follows from the characterization of the DLR-Gibbs
measures and from the hypothesis that
\begin{align*}
\nu(F)
=
\int_{\Omega} K_{n}(F,x)\, d\nu(x)
&=
\int_{\Omega} \left[\int_{\Omega} K_{n}(F,x)\, d\nu(x)\right] d\mu(y)
\\
&\geq
c
\int_{\Omega} \left[\int_{\Omega} K_{n}(F,y)\, d\nu(x)\right] d\mu(y)
\\
&=
c\int_{\Omega} \left[\int_{\Omega} K_{n}(F,y)\, d\mu(y)\right]d\nu(x)
\\
&=
c \mu(F).
\end{align*}
Using the monotone class theorem we may conclude that for
all $E\in\mathscr{F}$ we have $\nu(E)\geq c\mu(E)$.
In particular, $
0=\nu(\Omega\setminus E_0)\geq c\mu(\Omega\setminus E_0)$
therefore $\mu(E_0)=1$. Consequently $\mu$ is trivial on
$\cap_{j\in\mathbb{N}}\mathscr{T}_{j}$.
Hence another application of Theorem 7.7 (a) of
\cite{MR2807681} (or, see Corollary \ref{ap4}) ensures that $\mu$ is extreme.
\end{proof}

As a  consequence of this theorem
we prove the uniqueness of the eigenmeasures for the dual of Ruelle operator
associated to a potential $f:\Omega\to\mathbb{R}$ satisfying the following conditions:

\begin{itemize}
	\item (Walters)
	\vspace*{-0.88cm}
	\begin{align}\label{Walters-condition}
	\lim_{d(x,y)\to 0}
	\sup_{n\in\mathbb{N}}\
	\sup_{a\in M^n}\ |S_{n}(f)(ax)-S_{n}(f)(ay)|=0;
	\end{align}
	
	\item (Bowen)
	\vspace*{-0.88cm}
	\begin{align}\label{Bowen-condition}
	D
	\equiv
	\sup_{n\in\mathbb{N}}\
	\sup_{ \substack{x,y\in\Omega; \\ x_i=y_i, i=1,\ldots,n}}\ |S_{n}(f)(x)-S_{n}(f)(y)|<\infty.
	\end{align}
\end{itemize}

Of course, a potential $f$ satisfying the Walters condition satisfies
the Bowen condition. What we are calling here Bowen's condition is actually a
generalization to uncountable alphabets of classical Bowen's condition, see \cite{MR1783787}.
\begin{theorem}\label{teo-unicidade-d1}
Let $f$ be a continuous potential satisfying
\[
	D
	\equiv
	\sup_{n\in\mathbb{N}}\
	\sup_{ \substack{x,y\in\Omega; \\ x_i=y_i, i=1,\ldots,n}}\ |S_{n}(f)(x)-S_{n}(f)(y)|<\infty
\]
then the set
$
\mathscr{G}^{*}(f)
$
is a singleton and  $\mathscr{G}^{*}(f) = \mathscr{G}^{DLR}(f)$.
\end{theorem}

\begin{proof}
	Let $D$ be the constant as in the above theorem
	and $C$ a cylinder such that its basis is contained
	in the set $\{1,\ldots,p\}$, i.e., for every $n\geq p$
	we have $1_{C}(x_1\ldots x_n\sigma^{n}(z))=1_{C}(x_1\ldots x_n\sigma^{n}(y))$
	for all $y,z\in\Omega$ and $x_1,\ldots, x_n \in M$. We claim that for any choice of $y,z\in\Omega$ and for all
	$n\geq p$, we have
	\[
	e^{-2D}K_{n}(C,z)\leq K_{n}(C,y)\leq e^{2D}K_{n}(C,z).
	\]	
By definition of $D$ we have,
uniformly in $n\in\mathbb{N}$, $x,y,z\in\Omega$,
the following inequality
$
-D
\leq
S_n(f)(x_1\ldots x_n\sigma^{n}(z))-S_n(f)(x_1\ldots x_n\sigma^{n}(y))
\leq
D
$
which immediately imply the inequalities
$
\exp(-D)\exp(-S_n(f)(x_1\ldots x_n\sigma^{n}(z)))
\leq
\exp(-S_n(f)(x_1\ldots x_n\sigma^{n}(y)))
$
and
$	
\exp(-S_n(f)(x_1\ldots x_n\sigma^{n}(y)))		
\leq		
\exp(D)\exp(-S_n(f)(x_1\ldots x_n\sigma^{n}(z))).
$
Using theses two previous inequalities we get that
\begin{align}\label{eq-des-Lnfxy}
e^{-D}\mathscr{L}_{f}^{n}(1)(\sigma^{n}(z))
\leq
\mathscr{L}_{f}^{n}(1)(\sigma^{n}(y))
\leq
e^{D}\mathscr{L}_{f}^{n}(1)(\sigma^{n}(z))
\end{align}
and also
\begin{align*}
K_n(C,y)
=
\frac{\mathscr{L}_{f}^{n}(1_{C})(\sigma^{n}(y))}{\mathscr{L}_{f}^{n}(1)(\sigma^{n}(y))}
\leq
\frac{e^D \mathscr{L}_{f}^{n}(1_{C})(\sigma^{n}(z))}
{e^{-D}\mathscr{L}_{f}^{n}(1)(\sigma^{n}(z)) }
=
e^{2D}	K_n(C,z).
\end{align*}
Analogously we obtain $e^{-2D}K_{n}(C,z)\leq K_{n}(C,y)$ and
so the claim is proved.

Let $\mu$ and $\nu$ be distinct extreme measures in $\mathscr{G}^{DLR}(f)$.
Since we are assuming that $M$ is compact, it follows from Theorem 7.12 of \cite{MR2807681}
that there exist $y,z\in\Omega$ such that both
measures $\mu$ and $\nu$ are thermodynamic limits of
$K_n(\cdot,y)$ and $K_n(\cdot,z)$, respectively, when
$n\to\infty$.
Given an open cylinder set $C$ such that its basis is contained
in the set $\{1,\ldots,p\}$ there is an increasing sequence of closed cylinders
$C_1\subset C_2\subset\ldots$ such that for all $k\in\mathbb{N}$
the basis of $C_k$ is contained in the set $\{1,\ldots,p\}$,
and $\bigcup_{k\in\mathbb{N}} C_k = C$.
By Urysohn's lemma for each $k\in\mathbb{N}$ there is a continuous function
$\varphi_k:\Omega\to [0,1]$ such that $1_{C_k}\leq \varphi_k\leq 1_{C}$
and $\varphi_k\to 1_{C}$ pointwise. Since
$C$ and $(C_{k})_{k\in\mathbb{N}}$ have their basis contained in $\{1,\ldots,p\}$,
then the function $\varphi_k$
can be chosen as a continuous function depending only on its first $p$ coordinates.

By using the claim and a standard approximation arguments we get, for any fixed $k$,
the inequality
$
K_{n}(\varphi_{k},y)
\leq
e^{2D} K_{n}(\varphi_k,z)
$
for all $n\geq p$. By taking the limits, when $n$ goes to infinity
and next when $k$ goes to infinity we get
$
\mu(C) \leq  e^{2D} \nu(C).
$
Clearly the collection
$\mathscr{D}=\{E\in\mathscr{F}: \mu(E)\leq e^{2D\beta}\nu(E) \}$
is a monotone class. Since it contains the open cylinder sets,
which is stable under intersections,
we have that $\mathscr{D}=\mathscr{F}$.
Therefore $\mu \leq e^{2D\beta}\nu $, in particular $\mu\ll\nu$.
This contradicts the fact that two distinct extreme DLR-Gibbs measures
are mutually singular, therefore $\mathscr{G}^{DLR}(f)$ is a singleton
and by Lemma \ref{lemma-star-subset-DLR} we are done.
\end{proof}

This result generalizes two conditions for uniqueness
presented in two recent works by the authors when general
compact state space $M$ is considered, see
\cite{cioletti-silva-2016} and \cite{MR3377291}.
In fact, the above theorem
generalizes the H\"older, Walters (weak and stronger as introduced in
\cite{cioletti-silva-2016}) and Bowen conditions because
it can be applied for potentials defined
on $\Omega=M^{\mathbb{N}}$, where
the state space $M$ is any general compact metric space.

\subsection{Dobrushin Uniqueness Theorem}

In this section we prove an uniqueness theorem in the high temperature regime ($\beta$ small) for potentials not satisfying Bowen's condition in \eqref{Bowen-condition}. This result applies
for a very large class of potentials which live outside the H\"older, Walters and Bowen spaces.
Its proof is  based on the Dobrushin Uniqueness Theorem suitably adapted to our setting.

For each positive integer $n$ let $\Lambda_n\subset \mathbb{N}$
be a finite set such that $1\in \Lambda_n$. Denote by
$\pi_{\Lambda_n}:\Omega\to M^{\Lambda_n}$ the natural projection from $\Omega$
onto $ M^{\Lambda_n}$. For each $n\geq 1$ let $f_n:M^{|\Lambda_n|}\to \mathbb{R}$ be a continuous function and
suppose that $\sum_{n=1}^{\infty}\|f_n\|_{\infty}<+\infty$.
Consider the continuous potential $f:\Omega\to\mathbb{R}$ given by
\[
f(x) = \sum_{n=1}^{\infty} f_n\circ \pi_{\Lambda_n}(x).
\]

The next theorem is a version of Dobrushin's Theorem for Thermodynamic Formalism. We point out that in   \cite{Cioletti20176139} it was described a natural way to connect the classical  setting of Thermodynamic with interactions, specifications, etc... (which is more close to the classical setting of Statistical Mechanics). We will follow such point of view here.

\begin{theorem}
Let $f$ be as  above and suppose that
$\sum_{n\geq 1}  |\Lambda_n| \|f_n\|_{\infty}<+\infty$.
Then there exists $\beta_{D}\in (0,\infty)$ such that for any
$\beta<\beta_{D}$, the set $\mathscr{G}^{*}(\beta f)$ is a singleton.
\end{theorem}

\begin{proof}
Consider the interaction $\Phi\equiv (\Phi_{\Lambda})_{\Lambda\subset \mathbb{N}}$
given by: $\Phi_{\Lambda}\equiv 0$ if $\Lambda \neq k+\Lambda_n$ for some $k,n\in\mathbb{N}$;
$\Phi_{\Lambda_n}(x)= f_n\circ \pi_{\Lambda_n}(x)$ and
$\Phi_{k+\Lambda_n}(x)= f_n\circ \pi_{k+\Lambda_n}(x)$.

Note that
\[
H^{\Phi}_{\{1,\ldots,n\}}(x)
=
\sum_{\substack {\Lambda\cap \{1,\ldots,n\} \neq \emptyset\\ 0<|\Lambda|<+\infty }}
\Phi_{\Lambda}(x)
=
\sum_{i=1}^{n}
\sum_{\substack {\Lambda\ni i \\ 0<|\Lambda|<+\infty }}
\Phi_{\Lambda}(x)
=
S_n(f)(x).
\]
Let $\gamma^{\Phi}$ be the specification determined by $H^{\Phi}_{\Lambda}$.
Since the DLR Gibbs measures are completely determined in a cofinal collection
of volumes, we have that $\mathscr{G}(\gamma^{\beta\Phi})=\mathscr{G}^{DLR}(\beta f)$.

From the construction of $\Phi$ we have
\[
\sup_{i\in\mathbb{N}}\sum_{i\ni \Lambda} (|\Lambda|-1)\|\Phi_{\Lambda}\|_{\infty}
=
\sum_{n\geq 1} (|\Lambda_n|-1)\|f_{n}\|_{\infty}
<+\infty
\]
and it therefore follows from the Dobrushin Uniqueness Theorem (see \cite{MR2807681},  Theorem 8.7 and Proposition 8.8) that $\mathscr{G}(\gamma^{\beta\Phi})$ is a singleton whenever
\[
\beta < \beta_{D}\equiv \frac{2}{\sum_{n\geq 1} (|\Lambda_n|-1)\|f_{n}\|_{\infty}}.
\]

Since Lemma \ref{lemma-star-subset-DLR} ensures that for any continuous potential $f$
we have $\mathscr{G}^{*}(\beta f)\subset \mathscr{G}^{DLR}(\beta f)$ the result follows.
\end{proof}

We apply the above theorem in the following case. We take $E=\{-1,1\}$,
the a priori measure $p$ as the normalized counting measure and fix $0<\varepsilon<1$.
Consider the potential
\[
f(x) = \sum_{n=2}^{\infty}\frac{x_1x_{n}}{n^{1+\varepsilon}}.
\]

In the literature this potential is sometimes called Dyson potential (see \cite{Cioletti20176139}).
In this case, for any $\beta>0$ we have
\[
D
\equiv
\sup_{n\in\mathbb{N}}\
\sup_{ \substack{x,y\in\Omega; \\ x_i=y_i, i=1,\ldots,n}}\ |S_{n}(\beta f)(x)-S_{n}(\beta f)(y)|=\infty.
\]
The above equality implies that $f$ is not in H\"older, Walters and Bowen spaces.
Of course, this potential can be rewritten as
$
f(x) = \sum_{n=1}^{\infty} f_n\circ \pi_{\Lambda_n}(x),
$
by taking $\Lambda_n=\{1,n\}$ and $f_n\circ\pi_{\Lambda_n}(x) = x_1x_n/n^{1+\varepsilon}$.
Then
\[
\sum_{n\geq 1} (|\Lambda_n|-1)\|f_{n}\|_{\infty}
=
\sum_{n\geq 1} \frac{1}{n^{1+\varepsilon}} = \zeta(1+\varepsilon).
\]
Now applying the above theorem (taking $|\Lambda_n|=2$) we get that $\mathscr{G}^{*}(\beta f)$
is a singleton for any choice of $\beta< 2\zeta(1+\varepsilon)^{-1}$.

By taking $\Lambda_n=\{1,n,n+1\}$, $n \geq 2$,  one can get results for potentials of the form
$
f(x) = \sum_{n=2}^{\infty}x_1x_{n} x_{n+1}/n^{1+\varepsilon}
$
or more generally for potentials of the form
$
f(x) = \sum_{n=2}^{\infty}\, a_n\, x_1\,x_{n} \,x_{n+1},
$
where $\sum a_n$ is absolutely convergent.
In this case, to estimate the critical temperature, one considers the expression $\sum_{n\geq 1} 3\|f_{n}\|_{\infty}.$

\section{The Extension of the Ruelle Operator to the Lebesgue Space
$\pmb{L^{1}(\Omega,\mathscr{F},\nu_f)}$}\label{secao-ruelle-L1}

Let $f$ be a fixed continuous potential and $\nu_f$ the Borel probability measure obtained above. In this section we show how to construct a
bounded linear extension of the operator
$\mathscr{L}_f: C(\Omega)\to C(\Omega)$
acting on $L^{1}(\Omega,\mathscr{F},\nu_f)$,
by abusing notation also called $\mathscr{L}_f$, and under suitable assumptions
prove the existence of an almost surely non-negative eigenfunction
$\varphi_f\in L^1(\Omega,\mathscr{F},\nu_f)$
associated to the eigenvalue $\lambda_f$ constructed in the previous section.
\begin{proposition}\label{prop-ext-Lf-L1}
	Fix a continuous potential $f$ and let $\lambda_f$ and $\nu_f$ be the
	eigenvalue and eigenmeasure of $\mathscr{L}^{*}$,
	respectively.
	Then
	the Ruelle operator $\mathscr{L}_{f}:C(\Omega)\to C(\Omega)$
	can be uniquely extended to a bounded linear operator
	$
	\mathscr{L}_{f}:
	L^1(\Omega,\mathscr{F},\nu_f)
	\to
	L^1(\Omega,\mathscr{F},\nu_f)
	$.
	Moreover, this extension has operator norm
	$\|\mathscr{L}_{f}\|=\lambda_{f}$.
\end{proposition}
\begin{proof}
If $\varphi\in C(\Omega)$ then
$\varphi^{\pm}\equiv \max\{0,\pm\varphi\}\in C(\Omega)$.
Therefore, it follows from the positivity of the Ruelle operator and
\eqref{eq-dualidade} that
\begin{align*}
\|\mathscr{L}_{f}(\varphi)\|_{L^1}
&=
\int_{\Omega} |\mathscr{L}_{f}(\varphi^{+}-\varphi^{-})| \, d\nu_f
\leq
\int_{\Omega} |\mathscr{L}_{f}(\varphi^{+})|+ |\mathscr{L}_{f}(\varphi^{-})| \, d\nu_f
\\
&=
\int_{\Omega} \mathscr{L}_{f}(\varphi^{+}) +\mathscr{L}_{f}(\varphi^{-}) \, d\nu_f
=
\int_{\Omega} (\varphi^{+}+\varphi^{-}) \, d(\mathscr{L}^{*}_{f}\nu_f)
\\
&=
\lambda_{f}\int_{\Omega} (\varphi^{+}+\varphi^{-}) \, d\nu_f
=
\lambda_{f}\int_{\Omega} |\varphi| \, d\nu_f
\\
&=
\lambda_{f}\|\varphi\|_{L^1}.
\end{align*}
Since $\Omega$ is a compact Hausdorff space we have
\[
\overline{C(\Omega,\mathbb{R})}^{L^1(\Omega,\mathscr{F},\nu_f)}
=
L^1(\Omega,\mathscr{F},\nu_f),
\]
and  $\mathscr{L}_f$ therefore admits a unique continuous extension
to $L^1(\Omega,\mathscr{F},\nu_f)$.
By taking $\varphi\equiv 1$ it is easy to see that
$\|\mathscr{L}_{f}\|=\lambda_f$.
\end{proof}

\begin{proposition} \label{prop:extension to L1}
	For any fixed potential $f\in C(\Omega)$ we have that
	\[
	L^1(\Omega,\mathscr{F},\nu_f)
	=
	\Xi(f)
	\equiv
	\left\{
	\varphi\in  L^{1}(\Omega,\mathscr{F},\nu_{f}):
	\int_{\Omega} \mathscr{L}_{f}(\varphi)\, d\nu_{f} = \lambda_{f} \int_{\Omega}\varphi  d\nu_{f}
	\right\}.
	\]
\end{proposition}
\begin{proof}
From \eqref{eq-dualidade} it follows that
$C(\Omega)\subset \Xi(f)$. Let $\{\varphi_n\}_{n\in\mathbb{N}}$
be a sequence in $C(\Omega)$ such that
$\varphi_n \to \varphi$ in $L^{1}(\Omega,\mathscr{F},\nu_{f})$. Then
\[
\left|
\int_{\Omega} \varphi_n \, d\nu_{f} - \int_{\Omega} \varphi \, d\nu_{f}
\right|
\leq
\int_{\Omega} |\varphi_n-\varphi| \, d\nu_{f} \
\xrightarrow{n\to\infty} 0
\]
and using the boundedness of $\mathscr{L}_{f}$, we can
also conclude that
\[
\left|
\int_{\Omega}\!\! \mathscr{L}_{f}(\varphi_n) \, d\nu_{f}
-\!\!\!
\int_{\Omega}\!\! \mathscr{L}_{f}(\varphi) \, d\nu_{f}
\right|
\leq
\!
\int_{\Omega}\! |\mathscr{L}_{f}(\varphi_n-\varphi)| \, d\nu_{f}
\leq
\lambda_{f}\|\varphi_n-\varphi\|_{L^1}
\xrightarrow{n\to\infty} 0.
\]
By using the above
convergences and the triangular inequality
we can see that $\Xi(f)$ is closed subset of
$L^1(\Omega,\mathscr{F},\nu_f)$
. Indeed,
\begin{multline*}
\left|
\int_{\Omega} \mathscr{L}_{f}(\varphi) \, d\nu_{f}
-
\lambda_f\int_{\Omega} \varphi \, d\nu_{f}
\right|
\leq
\\
\left|
\int_{\Omega} \mathscr{L}_{f}(\varphi) \, d\nu_{f}
-
\int_{\Omega} \mathscr{L}_{f}(\varphi_n) \, d\nu_{f}
+
\int_{\Omega} \mathscr{L}_{f}(\varphi_n) \, d\nu_{f}
-
\lambda_f\int_{\Omega} \varphi \, d\nu_{f}
\right|
\end{multline*}
and the rhs goes to zero when $n\to\infty$ therefore $\varphi\in \Xi(f)$.
Since $C(\Omega,\mathbb{R})\subset \Xi(f)$ and $\Xi(f)$
is closed in $L^1(\Omega,\mathscr{F},\nu_f)$ we have that
\[
L^1(\Omega,\mathscr{F},\nu_f)
=
\overline{C(\Omega,\mathbb{R})}^{L^1(\Omega,\mathscr{F},\nu_f)}
\subset
\overline{\Xi(f)}^{L^1(\Omega,\mathscr{F},\nu_f)}
=\Xi(f)
\subset
L^1(\Omega,\mathscr{F},\nu_f).
\qedhere
\]
\end{proof}
The above proposition also implies that the extension of $\mathscr{L}_{f}/\lambda_f$ to $L^1(\Omega,\mathscr{F},\nu_f)$ can be identified with the transfer operator associated to a non-singular measure as formulated in the following corollary.

\begin{corollary} For any $\varphi \in L^1(\Omega,\mathscr{F},\nu_f)$ and $\psi \in L^\infty(\Omega,\mathscr{F},\nu_f)$,
\[ \int  {\lambda_f}^{-1} \mathscr{L}_{f}(\varphi) \; \psi d\nu_f =   \int  \varphi \; \psi\circ \sigma d\nu_f.\]
\end{corollary}

\begin{proof} Observe that
$\mathscr{L}_{f}(\varphi ) \psi = \mathscr{L}_{f}(\varphi\cdot \psi\circ \sigma)$
for  $\varphi, \psi \in C(\Omega)$.
In order to extend this property to the situation of the Corollary
that $(\psi_n)$ is a sequence of uniformly bounded functions in
$C(\Omega)$ which converges almost surely to
$\psi \in L^\infty(\Omega,\mathscr{F},\nu_f)$.
Then, for $\varphi \in C(\Omega)$,
it follows from bounded convergence and  Proposition \ref{prop:extension to L1} that
\begin{align*}
\int  \mathscr{L}_{f}(\varphi) \; \psi d\nu_f  & = \lim_{n \to \infty}  \int  \mathscr{L}_{f}(\varphi) \; \psi_n d\nu_f = \lim_{n \to \infty}  \int  \mathscr{L}_{f}(\varphi  \;  \psi_n \circ \sigma) d\nu_f  \\
& = \lim_{n \to \infty} \lambda_f \int  \varphi \; \psi_n   \circ \sigma) d\nu_f = \lambda_f \int  \varphi\; \psi \circ \sigma d\nu_f \\
&= \int  \mathscr{L}_{f}(\varphi \; \psi \circ \sigma)  d\nu_f
\end{align*}
The assertion then follows as in the proof of Prop. \ref{prop:extension to L1} by approximation of  $\varphi$ by continuous functions in $L^1(\Omega,\mathscr{F},\nu_f)$.
\end{proof}

As a corollary of a theorem by Lin (Th. 4.1 in \cite{Lin}, or Th. 1.3.3 in \cite{Aaronson}), the extremal DLR-measures hence can be characterized through the convergence of Ruelle's operator.

\begin{corollary} \label{Lin} The DLR-Gibbs measure $\nu_f$ is extremal if and only if for any $\varphi$ in the space $L^1(\Omega,\mathscr{F},\nu_f)$ with $\int_{\Omega} \varphi\, d\nu_f =0$, we have that
\[ \lim_{n \to \infty} \| \lambda_f^{-n}\mathscr{L}^n_{f}(\varphi) \|_1 =0 .\]
\end{corollary}

\section{Strong Convergence of Ruelle Operators}
\label{secao-convergencia-L1-operador-ruelle}

\begin{proposition}
\label{prop-convergencia-Lfn-Lf}
For any fixed potential $f\in C(\Omega)$ there is a sequence
$(f_n)_{n\in\mathbb{N}}$ contained in $C^{\gamma}(\Omega)$
such that $\|f_n-f\|_{\infty}\to 0$.
Moreover, for any eigenmeasure $\nu_f$ associated to the eigenvalue $\lambda_f$
we have that $\mathscr{L}_{f_n}$ has a unique continuous
extension to an operator defined on $L^1(\Omega,\mathscr{F},\nu_f)$
and moreover, in the uniform operator norm,
$\|\mathscr{L}_{f_n}-\mathscr{L}_{f}\|_{L^1(\Omega,\mathscr{F},\nu_f)}\to 0$,
when $n\to\infty$.
\end{proposition}

\begin{proof}
The first statement is a direct consequence of the Stone-Weierstrass Theorem.

For any $\varphi\in L^1(\Omega,\mathscr{F},\nu_f)$
the extension of $\mathscr{L}_{f_n}$ is given
by
$\mathscr{L}_{f_n}(\varphi)\equiv \mathscr{L}_{f}(\exp(f_n-f)\varphi)$
which is well-defined due to Proposition \ref{prop-ext-Lf-L1}.
From this proposition we also obtain
the inequality:
\begin{align*}
\int_{\Omega} |\mathscr{L}_{f_n}(\varphi)|\, d\nu_f
& =
\int_{\Omega} |\mathscr{L}_{f}(\exp(f_n-f)\varphi)|\, d\nu_f
\\
& \leq
\lambda_f \|\exp(f_n-f)\|_{\infty} \|\varphi\|_{L^1(\Omega,\mathscr{F},\nu_f)}
<\infty.
\end{align*}
Since the distance in the uniform operator norm
between $\mathscr{L}_{f_n}$ and $\mathscr{L}_{f}$ can be
 bounded from above by
\begin{align*}
\|\mathscr{L}_{f_n}-\mathscr{L}_{f}\|_{L^1(\Omega,\mathscr{F},\nu_f)}
&=
\sup_{0<\|\varphi\|_{L^1}\leq 1}
\int_{\Omega}
|\mathscr{L}_{f_n}(\varphi)-\mathscr{L}_{f}(\varphi)|\, d\nu_f
\\
&\leq
\sup_{0<\|\varphi\|_{L^1}\leq 1}
\int_{\Omega}
|\mathscr{L}_{f}(\exp(f_n-f)\varphi)-\mathscr{L}_{f}(\varphi)|\, d\nu_f
\\
&\leq
\lambda_{f}
\sup_{0<\|\varphi\|_{L^1}\leq 1}
\int_{\Omega}
|\varphi||(\exp(f_n-f)-1)|\, d\nu_f
\\
&\leq
\lambda_{f}
|\exp(\|f_n-f\|_{\infty})-1)|
\sup_{0<\|\varphi\|_{L^1}\leq 1}
\int_{\Omega}
|\varphi|\, d\nu_f,
\end{align*}
we can conclude that
$\|\mathscr{L}_{f_n}-\mathscr{L}_{f}\|_{L^1(\Omega,\mathscr{F},\nu_f)}\to 0$
as $n\to\infty$.
\end{proof}

\section{Existence of the Eigenfunctions}
\label{secao-existencia-auto-funcao}

We point out that for a given continuous potential $f$ there always exists eigenprobabilities $\nu_{f}$ whereas in some situations, there is no positive and continuous eigenfunction (see, for instance, \cite{MR3350377}).
We now will show the existence of
a non-trivial
eigenfunction of $\mathscr{L}_f$ in  $L^1(\Omega,\mathscr{F},\nu_{f})$ for potentials $f$ satisfying Bowen's condition. This partially extends a result of Walters (see \cite{MR1783787}) to our case with a possibly uncountable alphabet.

\medskip

In this section we consider sequences of Borel probability measures
$(\mu_{f_n})_{n\in\mathbb{N}}$ defined by
\begin{align}\label{def-mufn}
\mathscr{F}\ni E\mapsto
\mu_{f_n}(E)
\equiv
\int_{E}h_{f_n} d\nu_{f},
\end{align}

where $f_n\in C^{\gamma}(\Omega)$
satisfies $\|f_n-f\|_{\infty}\to 0$,
and $h_{f_n}$ is the unique eigenfunction of $\mathscr{L}_{f_n}$,
which is assumed to have $L^1(\Omega,\mathscr{F},\nu_f)$
norm one. Since $\Omega$ is compact we can also assume,
up to subsequence, that $\mu_{f_n}\rightharpoonup \mu\in \mathscr{P}(\Omega,\mathscr{F})$.

%

From the definition of $\mu_{f_n}$ we immediately have that
$\mu_{f_n}\ll \nu_f$.
Notice that, in such generality, it is \textbf{not} possible to
guarantee that $\mu\ll \nu_f$. When this is true the
Radon-Nikodym theorem ensures the existence of a non-negative
function $d\mu/d\nu_{f}\in L^1(\Omega,\mathscr{F},\nu_f)$
such that for all $E\in\mathscr{F}$ we have
\begin{align}\label{eq-radom-nikodym-mu-nuf}
\mu(E) = \int_{E} \frac{d\mu}{d\nu_f}\, d\nu_{f}.
\end{align}
In what follows we give sufficient conditions for this
Radon-Nikodym derivative to be an eigenfunction of $\mathscr{L}_f$.

\begin{theorem} Let $(f_n)$ be a sequence of Hölder continuous functions which converges uniformly to $f$,
and $\mu_{f_n}$ defined as  in \eqref{def-mufn}.
If $(h_{f_n})_{n\in\mathbb{N}}$ is a relatively compact subset of
$L^1(\Omega,\mathscr{F},\nu_f)$ then up to subsequence
$\mu_{f_n} \rightharpoonup \mu$, $\mu\ll \nu_{f}$ and
$\mathscr{L}_{f}(d\mu/d\nu_{f})=\lambda_f d\mu/d\nu_{f}$.
\end{theorem}

\begin{proof}
Without loss of generality we can assume that $h_{f_n}$ converges
to some non-negative function $h_f\in L^1(\Omega,\mathscr{F},\nu_f)$.
This convergence implies
\[
\left|
\int_{\Omega}\varphi h_{f_n}\, d\nu_{f}
-
\int_{\Omega}\varphi h_{f}\, d\nu_{f}
\right|
\to
0,
\quad \forall \varphi\in C(\Omega).
\]
Therefore $\mu_{f_n}\rightharpoonup \mu$ with $\mu\ll \nu_f$
and $d\mu/d\nu_{f}=h_f$ almost surely.

Let us show that this Radon-Nikodym derivative is a
non-negative eigenfunction for the Ruelle operator $\mathscr{L}_f$.
It follows from the triangular inequality  that
\begin{align*}
\|\mathscr{L}_{f}(h_f)-\lambda_fh_f \|_{L^1(\nu_f)}
\!\leq \!
\|\mathscr{L}_{f}(h_f)-\mathscr{L}_{f_n}(h_f) \|_{L^1(\nu_f)}
\!+\!
\|\mathscr{L}_{f_n}(h_f)-\lambda_fh_f \|_{L^1(\nu_f)}.
\end{align*}
Proposition \ref{prop-convergencia-Lfn-Lf} implies that
the first term tends to zero when $n$ tends to infinity, whereas
the second term can be estimated as follows.
\begin{align*}
\|\mathscr{L}_{f_n}(h_f)-\lambda_fh_f \|_{L^1(\nu_f)}
&\leq
\|\mathscr{L}_{f_n}(h_f-h_{f_n}+h_{f_n})-\lambda_fh_f \|_{L^1(\nu_f)}
\\
&\!\!\!\!\!\!
\leq
\|\mathscr{L}_{f_n}(h_f-h_{f_n})+ \lambda_{f_n}h_{f_n}-\lambda_fh_f \|_{L^1(\nu_f)}
\\
&\!\!\!\!\!\!\!\!\!\!
\leq
\|\mathscr{L}_{f_n}\|_{L^1(\nu)} \cdot
\|h_f-h_{f_n}\|_{L^1(\nu)}+
\|\lambda_{f_n}h_{f_n}-\lambda_fh_f \|_{L^1(\nu_f)}.
\end{align*}
Since $\sup_{n\in\mathbb{N}} \|\mathscr{L}_{f_n}\|_{L^1(\nu)}<+\infty$
and
$\|h_f-h_{f_n}\|_{L^1(\Omega,\mathscr{F},\nu_f)}\to 0$ as $n\to\infty$,
we have that the first term on the right hand side also goes to zero when $n$ tends to infinity.
The second term on the right hand side above is bounded by
\begin{align*}
\|\lambda_{f_n}h_{f_n}-\lambda_fh_f \|_{L^1(\nu_f)}
&\leq
\|\lambda_{f_n}h_{f_n}-\lambda_fh_{f_n} \|_{L^1(\nu_f)}
+
\|\lambda_{f}h_{f_n}-\lambda_fh_f \|_{L^1(\nu_f)}
\\
&=
|\lambda_{f_n}-\lambda_f|
+
|\lambda_{f}|\cdot \|h_{f_n}-h_f \|_{L^1(\nu_f)}.
\end{align*}
From Corollary \ref{cor-converg-lambdan}
and our assumption follows that the lhs above
can be made small if $n$ is big enough.
Piecing together all these estimates we
conclude that
$\|\mathscr{L}_{f}(h_f)-\lambda_fh_f \|_{L^1(\nu_f)}=0$
and therefore $\mathscr{L}_{f}(h_f)=\lambda_fh_f,\ \nu_f$\ a.s..
\end{proof}

\begin{theorem} Let $(f_n)$ be a sequence of Hölder continuous functions which converges uniformly to $f$,  $\mu_{f_n}$ as in \eqref{def-mufn} and suppose that
 that $\mu_{f_n} \rightharpoonup \mu$.
 If
$\mu\ll \nu_{f}$ and $h_{f_n}(x)\to d\mu/d\nu_{f}$ $\nu_{f}$-a.s.
then $\mathscr{L}_{f}(d\mu/d\nu_{f})=\lambda_f\, d\mu/d\nu_{f}$.
\end{theorem}

\begin{proof}
Notice that
\[
\int_{\Omega}|h_{f_n}|\, d\nu_f
=
1
=
\int_{\Omega}\left| \frac{d\mu}{d\nu_{f}}\right| \, d\nu_f
\quad
\text{and}
\ \ h_{f_n}(x)\to d\mu/d\nu_{f}\ \ \nu_{f}-\text{a.s.}.
\]
Scheff\'e's lemma implies that $h_{f_n}$ converges
to $d\mu/d\nu_{f}$ in the $L^1(\Omega,\mathscr{F},\nu_f)$ norm.
To finish the proof it is enough to apply the previous theorem.
\end{proof}

We now construct an eigenfunction for $\mathscr{L}_{f}$
without assuming convergence of $h_{f_n}$ neither in
$L^1(\Omega,\mathscr{F},\nu_f)$ nor almost surely.
We should remark that the next theorem applies even
when no convergent subsequence of $(h_{f_n})_{n\in\mathbb{N}}$
exists.

\begin{theorem}
Let $(h_{f_n})_{n\in\mathbb{N}}$ be a sequence of eigenfunctions
in the unit sphere of the Lebesgue space $L^1(\Omega,\mathscr{F},\nu_f)$,
where $(f_n)_{n\in\mathbb{N}}$ is a sequence of H\"older potentials converging uniformly to $f$.
If $\sup_{n\in\mathbb{N}}\|h_{f_n}\|_{\infty}<+\infty$,  then
$\limsup h_{f_n} \in L^1(\Omega,\mathscr{F},\nu_f)\setminus \{0\}$
and moreover $\mathscr{L}_{f}(\limsup h_{f_n}) = \lambda_f \limsup h_{f_n}$.
\end{theorem}
\begin{proof}
Since we are assuming that
$\sup_{n\in\mathbb{N}}\|h_{f_n}\|_{\infty}<+\infty$, it follows that
 $\limsup h_n\in L^1(\Omega,\mathscr{F},\nu_f)$.
For any fixed $x\in \Omega$ follows from this uniform
bound that the mapping
\[
M\ni a\mapsto \limsup_{n\to\infty} h_{f_n}(ax)
\]
is uniformly bounded and therefore integrable with
respect to the a-priori measure $\nu$
so we can apply the limsup version of the Fatou's lemma
to get the  inequality
\begin{align*}
\mathscr{L}_f(\limsup_{n\to\infty} h_{f_n})
&=
\int_{M} \exp(f(ax))\limsup_{n\to\infty} h_{f_n}(ax)\, dp(a)
\\
&=
\int_{M} \lim_{n\to\infty}\exp(f_n(ax))\limsup_{n\to\infty}
h_{f_n}(ax)\, dp(a)
\\
&=
\int_{M} \limsup_{n\to\infty}(\exp(f_n(ax)) h_{f_n}(ax))\, dp(a)
\\
&\geq
\limsup_{n\to\infty}
\int_{M} \exp(f_n(ax)) h_{f_n}(ax)\, dp(a)
\\
&=
\limsup_{n\to\infty}
\lambda_{f_n} h_{f_n}
\\
&=
\lambda_{f}
\limsup_{n\to\infty} h_{f_n}.
\end{align*}
These inequalities implies that $\limsup h_{f_n}$ is a super
solution to the eigenvalue problem. On the other hand,
we have proved that  $\|\mathscr{L}_f\|_{L^1(\nu_f)}=\lambda_f$.
This fact together with the previous inequality implies that
\[
\mathscr{L}_f(\limsup_{n\to\infty} h_{f_n})
=
\lambda_{f}
\limsup_{n\to\infty} h_{f_n}
\]
$\nu_f$-almost surely. It remains to prove that $\limsup h_{f_n}$ is non trivial.
Since we have $\sup_{n\in\mathbb{N}}\|h_{f_n}\|_{\infty}<+\infty$
follows that $\mu_{f_n} \rightharpoonup \mu \ll \nu_f$.
Indeed, for any open set $A\subset \Omega$,
weak convergence and the Portmanteau Theorem imply
that
\[
\mu(A)\leq \liminf_{n\to\infty} \int_{\Omega} 1_A h_{f_n}\, d\nu_{f}.
\]
Since $\nu_{f}$ is outer regular we have for any
$B\in\mathscr{F}$ that
$\nu_{f}(B)=\inf \{\nu_{f}(A): A\supset B, A\ \text{open} \}$.
From the previous inequality and uniform limitation of $h_{f_n}$
we get for any $B\subset A$ ($A$ open set) that
$\mu(B)\leq \mu(A)\leq \sup_{n\in\mathbb{N}}\|h_{f_n}\|_{\infty}\ \nu_{f}(A)$.
Taking the infimum over $A\supset B$, $A$ open, we have
$\mu(B) \leq \sup_{n\in\mathbb{N}}\|h_{f_n}\|_{\infty}\ \nu_{f}(B)$
and thus $\mu\ll \nu_f$.
By applying again the limit sup version of the Fatou Lemma we get that
\[
1
=
\int_{\Omega} \frac{d\mu}{d\nu_{f}}\, d\nu_f
=
\lim_{n\to\infty} \int_{\Omega} h_{f_n}\, d\nu_f
=
\limsup_{n\to\infty} \int_{\Omega} h_{f_n}\, d\nu_f
\leq
\int_{\Omega} \limsup_{n\to\infty} h_{f_n}\, d\nu_f,
\]
where the second equality comes from the definition of
the weak convergence.
\end{proof}
\medskip
We point out that the conditions $\sup_{n\in\mathbb{N}}\|h_{f_n}\|_{\infty}<+\infty$ and $(f_n)$ Hölder, $f_n \to f$ uniformly is not
satisfied for Hofbauer potentials.

All the previous theorems of this section
required informations about the eigenfunction.
Now we present an existence result that one can
check by using only the potential (via $\mathscr{L}^n_{f}$) and some estimates
on the maximal eigenvalue. We remark that a similar approach can be found in \cite{Denker-Kifer-Stadlbauer-2008} in the setting of random shift spaces with countable alphabets. In there, it was shown that the limes inferior defines a random eigenfunction.

\begin{theorem}\label{teo-eigenfunction-limsup}
Let $f$ be a continuous potential and $\lambda_f$ the eigenvalue of
$\mathscr{L}^*_f$ provided by Proposition \ref{prop-limite-pressao}.
If
\[
\sup_{n\in\mathbb{N}}
\left\| \mathscr{L}^n_{f}(1)/\lambda^n_{f}\right\|_{\infty}
<
+\infty,
\]
then,
$
\limsup_{n\to\infty}\mathscr{L}^n_{f}(1)/\lambda^n_{f}
$
is a non-trivial
eigenfunction of $\mathscr{L}_f$ in  $L^1(\Omega,\mathscr{F},\nu_{f})$ associated to $\lambda_f$.
\end{theorem}

\begin{proof}
The key idea is to prove that
$
\limsup_{n\to\infty}\mathscr{L}^n_{f}(1)/\lambda^n_{f}
$
is a super solution for the eigenvalue problem, since
it belongs to $L^1(\Omega,\mathscr{F},\nu_{f})$ it
has to be a sub solution and then it is in fact a solution.
Its non-triviality is based on the arguments given
in the previous proof and the weak convergence of
suitable sequence of probability measures.

The super solution part of the argument is again based on
the reverse Fatou Lemma as follows
\begin{align*}
\mathscr{L}_f(\limsup_{n\to\infty}
\mathscr{L}^n_{f}(1)/\lambda^n_{f}
)
&=
\int_{M} \exp(f(ax))\limsup_{n\to\infty}
\mathscr{L}^n_{f}(1)(ax)/\lambda^n_{f} \, dp(a)
\\
&\geq
\limsup_{n\to\infty}
\int_{M} \exp(f(ax))
\mathscr{L}^n_{f}(1)(ax)/\lambda^n_{f} \, dp(a)
\\
&=
\limsup_{n\to\infty}
\lambda_{f}\ \mathscr{L}^{n+1}_{f}(1)(x)/\lambda^{n+1}_{f}
\\
&=
\lambda_{f}
\limsup_{n\to\infty} \mathscr{L}^{n}_{f}(1)(x)/\lambda^{n}_{f}.
\end{align*}
The next step is to prove the non-triviality of this limsup.
From the definition of $\nu_f$ we can say that the following
sequence of probability measures is contained in $\mathscr{P}(\Omega,\mathscr{F})$:
\[
\mathscr{F}\ni E \mapsto \int_{E} \frac{\mathscr{L}^{n}_{f}(1)}{\lambda^n_{f}}
\, d\nu_{f}.
\]
Similarly, from the previous theorem we can ensure that all its cluster points
in the weak topology are absolutely continuous with respect to $\nu_f$.
Up to subsequence, we can get from another application of the Fatou
Lemma that
\[
1
=
\int_{\Omega} \frac{d\mu}{d\nu_{f}}\, d\nu_f
=
\lim_{n\to\infty} \int_{\Omega}
\frac{\mathscr{L}^{n}_{f}(1)}{\lambda^n_{f}}\, d\nu_f
\leq
\int_{\Omega} \limsup_{n\to\infty}
\frac{\mathscr{L}^{n}_{f}(1)}{\lambda^n_{f}}\, d\nu_f.
\]

\end{proof}

\begin{corollary}\label{cor-autofun-Bowen}
Let $f$ be a potential satisfying Bowen's condition and let $D$ as in \eqref{Bowen-condition}.
Then,
\[
h_{f}\equiv \limsup_{n\to\infty}\mathscr{L}^n_{f}(1)/\lambda^n_{f}
\]
is a non trivial $L^1(\Omega,\mathscr{F},\nu_{f})$
eigenfunction of $\mathscr{L}_f$ associated to $\lambda_f$
and $e^{-D}\leq h_{f}\leq e^{D}$.
\end{corollary}

\begin{proof}
Since we are assuming that the potential $f$ satisfies Bowen's condition
and $D$ is given by \eqref{Bowen-condition} it follows from \eqref{eq-des-Lnfxy} that uniformly in $n\in\mathbb{N}$ and $z,y\in\Omega$ we have
$
e^{-D}\mathscr{L}_{f}^{n}(1)(\sigma^{n}(z))
\leq
\mathscr{L}_{f}^{n}(1)(\sigma^{n}(y))
\leq
e^{D}\mathscr{L}_{f}^{n}(1)(\sigma^{n}(z)).
$
Replacing in this inequality $z$ by $a_1\ldots a_nz$ and similarly $y$ by $a_1\ldots a_ny$,
where $(a_1,\ldots, a_n)\in M^n$ is fixed, we obtain the following estimate
which holds for all $n\geq 1$ and $y,z\in \Omega$
\[
e^{-D}\mathscr{L}_{f}^{n}(1)(z)
\leq
\mathscr{L}_{f}^{n}(1)(y)
\leq
e^{D}\mathscr{L}_{f}^{n}(1)(z).
\]
By integrating the above inequality in $z$, with respect to
the eigenmeasure, we get
\[
e^{-D} \leq \frac{\mathscr{L}_{f}^{n}(1)(y)}{\lambda_{f}^n}\leq e^{D}.
\]
The conclusions then follow from the last inequality and Theorem \ref{teo-eigenfunction-limsup}.
\end{proof}

\begin{remark}
It is not possible to conclude from the above argument whether $h_f$ is
a continuous function. Similarly to the case of finite alphabet considered
in \cite{MR1783787} the best information we have so far about its regularity
is that this eigenfunction is at least
$L^{\infty}(\Omega,\mathscr{F},\nu_{f})$. In the context of Markov maps, this result
for example in 
\cite{ADU}. However, the continuity of this
eigenfunction as far as we know remains open, even in the finite alphabet setting.
\end{remark}

\section{Applications}
\label{secao-aplicacoes}

\subsubsection*{Weak Convergence of Eigenprobabilities}

In this section we
consider a continuous potential $f:\Omega\to\mathbb{R}$  or
an element of $C^{\gamma}(\Omega)$ for some $0\leq \gamma<1$.
We would like to get results for continuous potentials
via limits of H\"older potentials.

We choose a point in the state space $M$ and
for simplicity call it $0$. We denote by
$(f_n)_{n\in\mathbb{N}}\subset C^{\gamma}(\Omega)$
the sequence given by
$f_n(x)=f(x_1,\ldots,x_n,0,0,\ldots)$. Keeping the notation
of the previous sections, the eigenprobabilities
of $\mathscr{L}_{f_n}$ and $\mathscr{L}_{f}$ are also denoted
by $\nu_{f_n}$ and $\nu_f$, respectively.
Notice that $\|f-f_n\|_{\infty}\to 0$, when
$n\to\infty$ and, moreover, if $f$ is H\"older then
this convergence is exponentially fast.
We denote by $\mathcal{L}(C(\Omega))$
the space of all bounded operators from $C(\Omega)$ to
itself and for each $T\in \mathcal{L}(C(\Omega)) $
we use the notation $\|T\|_{C(\Omega)}$
for its operator norm.
The next lemma is inspired by  Proposition \ref{prop-convergencia-Lfn-Lf}.

\begin{lemma}
The sequence $(\mathscr{L}_{f_n})_{n\in\mathbb{N}}$
converges in the operator norm to the Ruelle operator $\mathscr{L}_{f}$, i.e.,
$\|\mathscr{L}_{f_n}-\mathscr{L}_{f}\|_{C(\Omega)}\to 0$, when $n\to\infty$.
\end{lemma}

\begin{proof}
For all $n\in\mathbb{N}$ we have
\begin{align*}
\|\mathscr{L}_{f_n}-\mathscr{L}_{f}\|_{C(\Omega)}
&=
\sup_{0<\|\varphi\|_{\infty}\leq 1}\
\sup_{x\in\Omega}
|\mathscr{L}_{f_n}(\varphi)(x)-\mathscr{L}_{f}(\varphi)(x)|
\\
&\leq
\sup_{0<\|\varphi\|_{\infty}\leq 1}\
\sup_{x\in\Omega}
|\mathscr{L}_{f}(\exp(f_n-f)\varphi)(x)-\mathscr{L}_{f}(\varphi)(x)|
\\
&\leq
\|\mathscr{L}_{f}\|_{C(\Omega)}
\sup_{0<\|\varphi\|_{\infty}\leq 1}\
\|\varphi\|_{\infty}
\|(\exp(f_n-f)-1)\|_{\infty}
\\
&\leq
\|\mathscr{L}_{f}\|_{C(\Omega)}
\|(\exp(f_n-f)-1)\|_{\infty}.
\end{align*}
\end{proof}

\begin{proposition}
Any cluster point, in the weak topology,
of the sequence $(\nu_{f_n})_{n\in\mathbb{N}}$ belongs to the
set $ \mathscr{G}^{*}(f)$.
\end{proposition}

\begin{proof}
By the previous lemma for any given $\varepsilon>0$
there is $n_0\in\mathbb{N}$ such that
if $n\geq n_0$ we have for all $\varphi\in C(\Omega)$ and
for all $x\in\Omega$ that
$
\mathscr{L}_{f_n}(\varphi)(x)-\varepsilon
<
\mathscr{L}_{f}(\varphi)(x)
<
\mathscr{L}_{f_n}(\varphi)(x)+\varepsilon.
$
From the duality relation of the Ruelle operator and
the weak convergence and Corollary \ref{cor-converg-lambdan}
we have that
\begin{align*}
\int_{\Omega} \varphi\, d(\mathscr{L}^*_{f}\nu)
=
\int_{\Omega} \mathscr{L}_{f}(\varphi)\, d\nu
&=
\lim_{n\to\infty}
\int_{\Omega} \mathscr{L}_{f}(\varphi)\, d\nu_{f_n}
\\
&<
\lim_{n\to\infty}
\int_{\Omega} \mathscr{L}_{f_n}(\varphi)\, d\nu_{f_n}+\varepsilon
\\
&=
\lim_{n\to\infty}
\int_{\Omega} \varphi\, d(\mathscr{L}_{f_n}\nu_{f_n})+\varepsilon
\\
&=
\lim_{n\to\infty}
\lambda_{f_n}
\int_{\Omega} \varphi\, d\nu_{f_n}+\varepsilon
\\
&=
\lambda_{f}
\int_{\Omega} \varphi\, d\nu+\varepsilon.
\end{align*}
We obtain analogous lower bound, with $-\varepsilon$
instead. Since
$\varepsilon>0$ is arbitrary it follows for any
$\varphi\in C(\Omega)$ that
\[
\int_{\Omega} \varphi\, d(\mathscr{L}^*_{f}\nu)
=
\lambda_{f}
\int_{\Omega} \varphi\, d\nu
\]
and therefore $\mathscr{L}^*_{f}\nu=\lambda_{f}\nu$.
\end{proof}

\begin{remark}
The above proposition for $f\in C^{\gamma}(\Omega)$
says that up to subsequence $\nu_{f_n}\rightharpoonup \nu_f$,
which is the unique eigenprobability of $\mathscr{L}^*_{f}$.
Therefore the eigenprobability $\nu_{f}$ inherits all
the properties of the sequence $\nu_{f_n}$ that are
preserved by weak limits.
\end{remark}

\subsubsection*{Constructive Approach for Equilibrium States for General Continuous Potentials}

Before proceed we should mention that Sarig in \cite{Sarig2015}
has also presented a construction of equilibrium measures for topologically
mixing topological Markov shifts (TMS).

\begin{lemma}\label{lema-construcao-muA}
For each $n\in\mathbb{N}$, let $f_n$ be the potential defined above and
$h_{f_n}$ the main eigenfunction of $\mathscr{L}_{f_n}$
associated to $\lambda_{f_n}$,
normalized so that $\|h_{f_n}\|_{L^1(\nu_{f_n})}=1$, where $\nu_{f_n}$
is the unique eigenprobability of $\mathscr{L}_{f_n}^{*}$.
Then there exist a $\sigma$-invariant Borel probability measure $\mu_{f}$
such that, up to taking subsequences,
\[
\lim_{n\to\infty} \int_{\Omega} \varphi h_{f_n}\,  d\nu_{f_{n}}
=
\int_{\Omega} \varphi\, d\mu_{f},
\quad \forall \varphi\in C(\Omega)
\]
\end{lemma}

\begin{proof}
It is well known that $h_{f_n}d\nu_{f_n}$
defines a $\sigma$-invariant Borel probability measure
and therefore any of its cluster point, in the weak topology
is a shift invariant probability measure.
\end{proof}

As observed in \cite{MR3377291}, when $M$ is uncountable, the
Kolmogorov-Sinai entropy is not suitable in the formulation of the variational problem.
In what follows we use the concept of entropy introduced in \cite{MR3377291}. This
entropy is defined for each probability measure $\mu$ by
\begin{align}\label{def-entropy-variacional}
h_{\mu}(\sigma)
\equiv
\inf_{g \in C^{\alpha}(\Omega,\mathbb{R})}
\left\lbrace - \int_{\Omega} g \  d\mu + \log \lambda_g \right\rbrace.
\end{align}

Note that this entropy depends on the choice of the a priori measure and
similar ideas are employed in Statistical Mechanics to study translation
invariant Gibbs measures of continuous spin systems on the lattice,
see \cite{MR2807681,MR1241537} and references therein.

\begin{theorem}[Equilibrium States]\label{teo-estados-equilibrio}
Let $f:\Omega\to\mathbb{R}$ be a continuous potential and
$(f_n)_{n\in\mathbb{N}}$ a sequence
of H\"older potentials such that $\|f_n-f\|_{\infty}\to 0$,
when $n\to\infty$. Then any probability measure $\mu_{f}$
as constructed in the Lemma \ref{lema-construcao-muA}
is an equilibrium state for $f$.
\end{theorem}

\begin{proof}
Given any $\varepsilon>0$ there is $n_0\in\mathbb{N}$ so that
if $n\geq n_0$ then $-\varepsilon<f-f_n<\varepsilon $.
We know that the equilibrium measure $\mu_{f_n}$ for $f_n$ is given by
$\mu_{f_n}=h_{f_n}\nu_{f_n}$ and therefore, we have
that
\begin{align*}
\sup_{\rho\in\mathscr{P}_{\sigma}(\Omega,\mathscr{F})}
\left\{ h(\rho) + \int_{\Omega} f\, d\rho \right\}
&=
\sup_{\rho\in\mathscr{P}_{\sigma}(\Omega,\mathscr{F})}
\left\{ h(\rho) + \int_{\Omega} (f-f_n)\, d\rho +\int_{\Omega} f_n\, d\rho \right\}
\\
&<
\varepsilon+
\sup_{\rho\in\mathscr{P}_{\sigma}(\Omega,\mathscr{F})}
\left\{ h(\rho) +\int_{\Omega} f_n\, d\rho \right\}
\\
&=
\varepsilon+
h(\mu_{f_n}) +\int_{\Omega} f_n\, d\mu_{f_n}.
\end{align*}
Since the entropy defined by \eqref{def-entropy-variacional} is upper semi-continuous
and $\mu_{f_n}\rightharpoonup \mu_{f}$ it follows that
for some $n_1\in\mathbb{N}$ and $n\geq n_1$ we have
\[
h(\mu_{f_n})<h(\mu_{f})+\varepsilon.
\]
Using again the uniform convergence of $f_n$ to $f$
and the weak convergence of $\mu_{f_n}$ to $\mu_{f}$, for
some $n_2\in\mathbb{N}$ and $n\geq n_2$ we get
\begin{align*}
\int_{\Omega} f_n\, d\mu_{f_n}
=
\int_{\Omega} f_n-f\, d\mu_{f_n}
+
\int_{\Omega} f\, d\mu_{f_n}
<
2\varepsilon
+
\int_{\Omega} f \, d\mu_{f}.
\end{align*}
Using the previous three inequalities we get
for $n\geq \max\{n_0,n_1,n_2\}$
\begin{align*}
\sup_{\rho\in\mathscr{P}_{\sigma}(\Omega,\mathscr{F})}
\left\{ h(\rho) + \int_{\Omega} f\, d\rho \right\}
<
4\varepsilon+
h(\mu_f) +\int_{\Omega} f\, d\mu_{f}.
\end{align*}
Since $\varepsilon>0$ is arbitrary follows from the
definition of the supremum and above inequality that
\begin{align*}
\sup_{\rho\in\mathscr{P}_{\sigma}(\Omega,\mathscr{F})}
\left\{ h(\rho) + \int_{\Omega} f\, d\rho \right\}
=
h(\mu_f) +\int_{\Omega} f\, d\mu_{f}
\end{align*}
and therefore $\mu_f$ constructed in Lemma \ref{lema-construcao-muA}
is an equilibrium state.
\end{proof}

\begin{corollary}
For any continuous potential $f:\Omega\to\mathbb{R}$ we have that
\[
\log \lambda_{f} =
\sup_{\rho\in\mathscr{P}_{\sigma}(\Omega,\mathscr{F})}
\left\{ h(\rho) + \int_{\Omega} f\, d\rho \right\}.
\]
\end{corollary}

\begin{proof}
Consider the H\"older approximations
$(f_n)_{n\in\mathbb{N}}$ of $f$ as above. Then
for any given $\varepsilon>0$ and $n$ large enough
we have
\begin{align*}
\log \lambda_{f_n}-\varepsilon
&=
h(\mu_{f_n}) +\int_{\Omega} f_n\, d\mu_{f_n}
-\varepsilon
\\
&<
\sup_{\rho\in\mathscr{P}_{\sigma}(\Omega,\mathscr{F})}
\left\{ h(\rho) + \int_{\Omega} f\, d\rho \right\}
\\
&<
\varepsilon+
h(\mu_{f_n}) +\int_{\Omega} f_n\, d\mu_{f_n}
\\
&=
\varepsilon +\log \lambda_{f_n}.
\end{align*}
Since $\lambda_{f_n}\to \lambda_{f}$ it follows from the
above inequality that
\[
\sup_{\rho\in\mathscr{P}_{\sigma}(\Omega,\mathscr{F})}
\left\{ h(\rho) + \int_{\Omega} f\, d\rho \right\}
=
\log\lambda_{f}.
\qedhere
\]
\end{proof}

\subsubsection*{Necessary and Sufficient Conditions for the Existence of $\pmb{L}^1$ Eigenfunctions}
\begin{theorem}
Let $\nu\in\mathscr{G}^*(f)$.
The Ruelle operator has a non-negative eigenfunction
$h\in L^1(\nu)$ if, and only if,
there exists
$\mu\in\mathscr{P}_{\sigma}(\Omega,\mathscr{F})$ such that
$\mu\ll \nu$.

\end{theorem}

\begin{proof}
We first assume that there is $\mu\in\mathscr{P}_{\sigma}(\Omega,\mathscr{F})$
so that $\mu\ll \nu$. In this case we claim that
\[
\mathscr{L}_{f}\!\left(\frac{d\mu}{d\nu} \right)
=
\lambda_{f}\frac{d\mu}{d\nu}.
\]
Indeed, for any continuous function $\varphi$ we have
\begin{align*}
\int_{\Omega} \varphi \mathscr{L}_{f}\left( \frac{d\mu}{d\nu} \right)\, d\nu
&=
\int_{\Omega} \mathscr{L}_{f}\left(\varphi\circ\sigma \cdot \frac{d\mu}{d\nu}  \right)\, d\nu
\\
&=
\lambda_f\int_{\Omega}\varphi\circ\sigma\cdot \frac{d\mu}{d\nu}\, d\nu
=
\lambda_f\int_{\Omega}\varphi\circ\sigma\cdot\, d\mu
\\
&=
\lambda_f\int_{\Omega}\varphi\, d\mu
=
\lambda_f\int_{\Omega}\varphi \cdot \frac{d\mu}{d\nu}\, d\nu.
\end{align*}

Conversely, suppose that $h\in L^{1}(\nu)$ is a non-negative
eigenfunction for the Ruelle operator associated to the main
eigenvalue and normalized so that $\int_{\Omega} h\, d\nu=1$.
Define the probability measure $\mu =hd\nu$.
Then for any $\varphi\in C(\Omega)$ we have
\begin{align*}
\lambda_f \int_{\Omega} \varphi\, d\mu
&=
\lambda_f \int_{\Omega} \varphi h\, d\nu
=
\int_{\Omega} \varphi\, \mathscr{L}_fh \ d\nu
\\
&=
\int_{\Omega} \mathscr{L}_f(\varphi\circ\sigma \cdot h) \ d\nu
=
\lambda_{f}\int_{\Omega} \varphi\circ\sigma \cdot h \ d\nu
\\
&=
\lambda_{f}\int_{\Omega} \varphi\circ\sigma\, d\mu
\end{align*}
and therefore $\mu\in\mathscr{P}_{\sigma}(\Omega,\mathscr{F})$ and
$\mu\ll \nu$.
\end{proof}

\subsubsection*{Continuous Potentials not Having Continuous Eigenfunctions}

Now we assume that the state space $M=\{-1,1\}$ and the a priori measure
is a Bernoulli measure, which we denote by $\kappa$.
Let $\rho$ be the infinite product measure
$\rho=\prod_{i\in\mathbb{N}}\kappa$.
Consider the continuous potential $f$ given by
$f(x)=\sum_{n\geq 1} (x_n/n^{\gamma})$, where $3/2<\gamma\leq 2$.
For each $n\in\mathbb{N}$ set
$
\alpha_n = \zeta(\gamma)- \sum_{j=1}^{n} n^{-\gamma}.
$
From Theorem 5.1 in \cite{JLMS:JLMS12031}
we have that the main eigenvalue for $\mathscr{L}_{f}$ is
$\lambda_f=2\cosh(\zeta(\gamma))$ and there is a $\mathscr{F}$-measurable
set $\Omega_0\subset\Omega$
satisfying $\rho(\Omega_0)=1$ and such that
for all $x\in \Omega_0$ the following function
\[
x\mapsto h_f(x)\equiv
\exp(\alpha_1 x_1+ \alpha_2 \, x_2 + \alpha_3 x_3 +   \ldots+\,\alpha_n x_n\,+\ldots)
\]
is well defined.

From Theorem 6.1 item (iv) in \cite{JLMS:JLMS12031}
it follows that $h_f$ is the unique eigenfunction
associated to $\lambda_f$ and it is not an element of
$L^{\infty}(\Omega,\mathscr{F},\rho)$ which implies that $h_f\notin C(\Omega)$.
On the other hand, from Theorem 5.1 in \cite{JLMS:JLMS12031} follows
that  $h_f\in L^1(\Omega,\mathscr{F},\nu_f)$.

\section{Appendix}

On this appendix we adapt some results  from the reference \cite{MR2807681} to the present setting.
Let $\mathscr{L}_f $ be the Ruelle operator of a
continuous potential $f$ and  for each $n\in\mathbb{N}$,
$x\in\Omega$ and $E\in\mathscr{F}$, consider the probability kernel
$K_n:\mathscr{F}\times\Omega\to [0,1]$ given by the expression
$$
K_n(E,x)\equiv
\frac{\mathscr{L}_{f}^{n}(1_E)(\sigma^n(x)) }
{\mathscr{L}^{n}_{f}(1)(\sigma^n(x))}.
$$

\begin{proposition} \label{ap1} Suppose $\mu \in \mathscr{G}^{DLR}(f)$, then
for all $n\in\mathbb{N}$
\[
\mathscr{A}_n(\mu)
=
\{ E\in\mathscr{F}: K_n(E,\omega)= 1_E(\omega) \ \mu \ \mathrm{a.s.} \}
\quad \mathrm{is\ a}\  \sigma-\mathrm{algebra}.
\]

\end{proposition}

\begin{proof}
Since $K_n(\Omega,\omega)=1=1_{\Omega}(\omega)$
we get that $\Omega\in A_n(\mu)$. For the empty set the proof is trivial.

Now suppose that $(E_j)_{j\in\mathbb{N}}$ is a disjoint collection
of elements of $\mathscr{A}_n(\mu)$. Then, for all $\omega$ we get
$K_n(\cup_{j\in\mathbb{N}}E_j,\omega)= \sum_{j\in\mathbb{N}} K_n(E_j,\omega)$.
Note that $\mu$-a.e.  $K_n(E_j,\omega) =1_{E_j}(\omega)$
for all $j\in\mathbb{N}$, because $E_j\in \mathscr{A}_n(\mu)$.
Clearly,
$1_{\cup_{j\in\mathbb{N}} E_j}(\omega) = \sum_{j\in\mathbb{N}} 1_{E_j}(\omega)$,
then by using that the intersection of sets of measure one has measure one,
we get that
$K_n(\cup_{j\in\mathbb{N}}E_j,\omega)=1_{\cup_{j\in\mathbb{N}} E_j}(\omega)$,
$\mu$-a.e..

Note that $\mathscr{A}_n(\mu) $ is closed by the complement operation.
Indeed,
for all $\omega\in\Omega$ and $E\in \mathscr{A}_n(\mu)$
we have that
$K_{n}(E^c,\omega)=1-K_n(E,\omega)=1-1_{E}(\omega)=1_{E^c}(\omega)$.

Since we have shown that $\mathscr{A}_n(\mu)$ is closed under
denumerable disjoint unions then the remaining task is
to show that $\mathscr{A}_n(\mu)$ is closed under finite intersections.
Then it will follow that $\mathscr{A}_n(\mu)$ is closed
under any denumerable union.
Suppose that $E,F\in \mathscr{A}_n(\mu)$.
By the monotonicity of the measure
we have $\mu$-a.e that
\begin{align*}
K_n(E\cap F,\omega)
&\leq
\min\{ K_n(E,\omega),K_n(F,\omega) \}
\\
&=
\min\{ 1_{E}(\omega),1_{F}(\omega) \}
\\
&=
1_{E\cap F}(\omega).
\end{align*}
By using the hypothesis we get that
\begin{align*}
\int_{\Omega} [1_{E\cap F}- K_n(E\cap F,\cdot)]\, d\mu
&=
\int_{\Omega} 1_{E\cap F}\, d\mu
-
\int_{\Omega} K_n(E\cap F,\cdot)\, d\mu
\\
&=
\mu(E\cap F)
-
\int_{\Omega} K_n(E\cap F,\cdot)\, d\mu
\\
&=
\mu(E\cap F)
-
\mu(E\cap F)
\\
&=
0.
\end{align*}
From the previous inequality we known that the integrand
in the left hand side of the above is non-negative.
So it has to be zero $\mu$-a.e..
Therefore, $K_n(E\cap F,\omega)=1_{E\cap F}(\omega)$, $\mu$-a.e..
and finally we get that $\mathscr{A}_n(\mu)$ is closed for
finite intersections.
Therefore,
$\mathscr{A}_n(\mu)$ is a $\sigma$-algebra.
\end{proof}

\begin{proposition} \label{ap2}
Given a function $g:\Omega\to [0,\infty)$ we get the equivalence:
\medskip

1- $\int_{\Omega} K_n(E,\cdot )g\, d\mu  = \int_{\Omega} 1_{E} g\, d\mu$
for all  $E\in\mathscr{F}$,
\medskip

2- The function $g$ is measurable with respect to the sigma-algebra $\mathscr{A}_n(\mu)$.
\end{proposition}

\begin{remark}
In \cite{MR2807681} the condition 1 is denoted by
$(g\mu)K_n=g\mu$, where  $g\mu$ is the measure defined by
$E\mapsto \int_{\Omega} 1_E\, g\, d\mu$.
This condition is equivalent to say that
$g\mu$ is compatible with $K_n$.
\end{remark}
\begin{proof}
First we will prove that $\mathit{1}\, \Longrightarrow \mathit{2}$.
This follows from the following claim:
for all $g:\Omega\to [0,\infty)$ for which the condition
\textit{1} holds, we have $\{g\geq c\}\in \mathscr{A}_n(\mu)$,
for any $c\in\mathbb{R}$. Indeed, the identity
$1_{\{g<c\}}=1-1_{\{g\geq c \}}$  implies
\begin{multline*}
\int_{ \{g<c\} } K_n( 1_{\{g\geq c\}},\omega) g(\omega)\, d\mu(\omega)
\\[0.3cm]
=
\int_{ \Omega } K_n( 1_{\{g\geq c\}},\omega) g(\omega)\, d\mu(\omega)
-
\int_{ \Omega } 1_{\{g\geq c\}}(\omega) g(\omega) K_n( 1_{\{g\geq c\}},\omega)
\, d\mu(\omega).
\end{multline*}
By using the condition \textit{1} in the first expression of rhs we get
\begin{multline*}
\int_{ \{g<c\} } K_n( 1_{\{g\geq c\}},\omega) g(\omega)\, d\mu(\omega)
\\[0.3cm]
=
\int_{ \Omega } 1_{\{g\geq c\}}(\omega) g(\omega)\, d\mu(\omega)
-
\int_{ \Omega } 1_{\{g\geq c\}}(\omega) g(\omega) K_n( 1_{\{g\geq c\}},\omega)
\, d\mu(\omega)
\\[0.3cm]
=
\int_{ \Omega } 1_{\{g\geq c\}}(\omega) g(\omega) (1-K_n( 1_{\{g\geq c\}},\omega))
\, d\mu(\omega).
\end{multline*}

Now, we will use  the two inequalities
$1_{\{g\geq c\}}(\omega) g(\omega)\geq c\cdot 1_{\{g\geq c\}}(\omega)$
and $K_n( 1_{\{g\geq c\}},\omega)\leq 1$, in the above expression,
to get
\begin{align*}
\int_{ \{g<c\} }& K_n( 1_{\{g\geq c\}},\omega) g(\omega)\, d\mu(\omega)
\\[0.3cm]
&=
\int_{ \Omega } 1_{\{g\geq c\}}(\omega) g(\omega) (1-K_n( 1_{\{g\geq c\}},\omega))
\, d\mu(\omega)
\\[0.3cm]
&\geq
c\int_{ \Omega } 1_{\{g\geq c\}}(\omega) (1-K_n( 1_{\{g\geq c\}},\omega))
\, d\mu(\omega)
\\[0.3cm]
&=
c\int_{ \Omega } 1_{\{g\geq c\}}(\omega) \, d\mu(\omega)
-
c\int_{ \Omega } 1_{\{g\geq c\}}(\omega)K_n( 1_{\{g\geq c\}},\omega) \, d\mu(\omega)
\\[0.3cm]
&\underset{=}{{\scriptstyle(cond\ 1)}}\ \
c\int_{ \Omega } K_n( 1_{\{g\geq c\}},\omega) \, d\mu(\omega)
-
c\int_{ \Omega } 1_{\{g\geq c\}}(\omega)K_n( 1_{\{g\geq c\}},\omega) \, d\mu(\omega)
\\[0.3cm]
&\underset{=}{ {\scriptstyle (1_{\{g<c\}}=1-1_{\{g\leq c\}})} }\ \
c\int_{ \{g<c\} } K_n( 1_{\{g\geq c\}},\omega) \, d\mu(\omega).
\end{align*}
Now, the two extremes of the above inequality give us
\[
\int_{ \{g<c\} } (g-c)\, K_n( 1_{\{g\geq c\}},\omega)\ d\mu(\omega)\geq 0.
\]
Therefore,
$
1_{\{g<c\}}(\omega) K_n( 1_{\{g\geq c\}},\omega) = 0 \
$
$\mu$-a.e.. From this follows that
\begin{align*}
K_n( 1_{\{g\geq c\}},\omega)
&=
1_{\{g\geq c\}}(\omega)\ K_n( 1_{\{g\geq c\}},\omega)
+
1_{\{g< c\}}(\omega)\ K_n( 1_{\{g\geq c\}},\omega)
\\
&=
1_{\{g\geq c\}}(\omega)\ K_n( 1_{\{g\geq c\}},\omega)
\\
&\leq
1_{\{g\geq c\}}(\omega).
\end{align*}
By another application of the condition \textit{1} we get
\[
\int_{\Omega} 1_{\{g\geq c\}}(\omega)- K_n( 1_{\{g\geq c\}},\omega)\, d\mu =0
\]
and, then from the last inequality we obtain the  $\mu$-a.e.
equality
$1_{\{g\geq c\}}(\omega)= K_n( 1_{\{g\geq c\}},\omega)$.
This means that $\{g\geq c\}\in \mathscr{A}_n(\mu)$
and so $g$ is $\mathscr{A}_n(\mu)$-mensurable.

\medskip

Now we will show that $2\, \Longrightarrow\, 1$.
Suppose  $g$ is $\mathscr{A}_n(\mu)$-mensurable.
First we will show that $2\, \Longrightarrow\, 1$
holds when $g=1_{F}$, for some $F\in \mathscr{A}_n(\mu)$.
To prove this claim, it only remains to verify that
\begin{equation}\label{prova-func-indicadora}
\int_{\Omega} 1_F\cdot K_n(E,\cdot)\, d\mu
=
\int_{\Omega} 1_F\cdot 1_{E}\, d\mu,
\qquad \forall \ E\in\mathscr{F}.
\end{equation}
Note that for any $ E\in\mathscr{F}$ we have
\begin{align*}
\int_{\Omega} 1_F\cdot K_n(E,\cdot)\, d\mu
&=
\int_{\Omega} 1_F\cdot K_n(E\cap F,\cdot)\, d\mu
+
\int_{\Omega} 1_F\cdot K_n(E\cap F^c,\cdot)\, d\mu
\\
&\leq
\int_{\Omega} K_n(E\cap F,\cdot)\, d\mu
+
\int_{\Omega} 1_F\cdot K_n(F^c,\cdot)\, d\mu
\\
&\hspace*{-0.5cm}\underset{=}{\scriptstyle(\mathrm{Hip. \,on\, K_n})}
\int_{\Omega} 1_{E\cap F}\, d\mu
+
\int_{\Omega} 1_F\cdot K_n(F^c,\cdot)\, d\mu
\\
&\hspace*{-0.5cm}\underset{=}{\scriptstyle(F\in \mathscr{A}_n(\mu))}
\int_{\Omega} 1_{E\cap F}\, d\mu
+
\int_{\Omega} 1_F\cdot 1_{F^c} \, d\mu
\\
&=
\int_{\Omega} 1_{E}\cdot 1_{F}\, d\mu.
\end{align*}

By a similar argument we can show that
\begin{align*}
\int_{\Omega} 1_{F}\cdot K_n(E^c,\cdot)\, d\mu
\leq
\int_{\Omega} 1_{E^c}\cdot 1_{F}\, d\mu.
\end{align*}
Since
\begin{align*}
\int_{\Omega} 1_{F}\cdot K_n(E,\cdot)\, d\mu
+
\int_{\Omega} 1_{F}\cdot K_n(E^c,\cdot)\, d\mu
&=
\mu(F)
\\[0.3cm]
&=
\int_{\Omega} 1_{F}\cdot 1_{E}\, d\mu
+
\int_{\Omega} 1_{F}\cdot 1_{E^c}\, d\mu
\end{align*}
it follows from the two last inequalities that
\[
\int_{\Omega} 1_F\cdot K_n(E,\cdot)\, d\mu
=
\int_{\Omega} 1_F\cdot 1_{E}\, d\mu,
\qquad \forall \ E\in\mathscr{F}.
\]
The above identity extends by linearity for simple functions.
By taking a sequence of simple functions $\varphi_k \uparrow f$,
and using the monotone convergence theorem we get for any
$\mathscr{A}_n(\mu)$-measurable function $g$ that
\[
\int_{\Omega} g\cdot K_n(E,\cdot)\, d\mu
=
\int_{\Omega} g\cdot 1_{E}\, d\mu,
\qquad \forall \ E\in\mathscr{F}.
\]

\end{proof}

\medskip
It follows from last proposition that if $g$ is measurable with respect to the sigma-algebra $\mathscr{A}_n(\mu)$ for all $n\in \mathbb{N}$, and $\mu\in \mathscr{G}^{DLR}(f)$, $f$ continuous, then $g\mu$ is also in $\mathscr{G}^{DLR}(f)$

\begin{corollary} \label{ap4}

Given $\mu\in \mathscr{G}^{DLR} (f)$ define
$\mathscr{A}(\mu) \equiv \bigcap_{n\in\mathbb{N}} \mathscr{A}_n(\mu)$.
Then,  $\mu$ is extreme in $\mathscr{G}^{DLR}(f)$, if an only if,
$\mu$ is trivial on  $\mathscr{A}(\mu)$.

\end{corollary}

\begin{proof}
Suppose that there exists $F\in\mathscr{A}(\mu)$ such that $0<\mu(F)<1$
and consider the following probability measures
\[
\begin{array}{c}
\displaystyle\mathscr{F}\ni E\mapsto \nu(E) = \mu(E|F) = \int_{\Omega} \frac{1}{\mu(F)}1_{E}1_{F}\, d\mu,
\\[0.5cm]
\displaystyle\mathscr{F}\ni E\mapsto \gamma(E) = \mu(E|F^c) = \int_{\Omega} \frac{1}{\mu(F^c)}1_{E}1_{F^c}\, d\mu.
\end{array}
\]
Clearly $\nu\neq \gamma$ and moreover
\begin{equation}\label{comb-convexa}\mu  = \mu(F)\nu +(1-\mu(F))\gamma.\end{equation}

The last proposition guarantees that both $\nu$ and $\gamma$
belong to $\mathscr{G}^{DLR}(f)$.
Indeed,  in last proposition take $f$
as $(1/\mu(F))\cdot 1_{F}$ and  $(1/\mu(F^c))\cdot 1_{F^c}$, respectively (these functions are
$\mathscr{A}_n(\mu)$-measurable for $n\in\mathbb{N}$).
However the existence of the non trivial convex combination \eqref{comb-convexa},
of two elements in $\mathscr{G}^{DLR}(f)$, is a contradiction.
Therefore, any set $F\in\mathscr{A}(\mu)$ has the $\mu$ measure
zero or one.

Conversely, suppose that $\mu$ is trivial on $\mathscr{A}(\mu)$
and at same time expressible as $\mu = \lambda\nu +(1-\lambda)\gamma$,
with $0<\lambda<1$ and $\nu,\gamma\in \mathscr{G}(f)$.

Note that  $\nu \ll \mu$ and then from
Radon-Nikodym Theorem we get that
$\nu(E) = \int_{\Omega} 1_{E}f\, d\mu$  for some measurable function $f\geq 0$. Once more by the equivalence
$1\, \Longleftrightarrow\, 2$
we get that  $f$ is $\mathscr{A}_n(\mu)$-measurable for all $n\in\mathbb{N}$
(recall that $\nu\in\mathscr{G}^{DLR}(f)$).
Since we assumed that $\mu$ is trivial on $\mathscr{A}(\mu)$
we get that both integrals below are always equals to each other
being zero or one
\[
\int_{\Omega} 1_{F}f\, d\mu = \int_{\Omega} 1_F d\mu,
\qquad \forall F\in \mathscr{A}(\mu).
\]

As the equality is valid for all $F\in \mathscr{A}(\mu)$
and  $f$ is $\mathscr{A}_n(\mu)$-measurable
we can conclude that $f=1$ $\mu$-a.e..
Therefore, $\mu=\nu$ and consequently $\gamma=\mu$.
So $\mu$ is a extreme point of $\mathscr{G}^{DLR}(f) $.
\end{proof}

\medskip

It follows from last corollary that if $\mathscr{G}^{DLR}(f)$
has only one element  $\mu$, then $\mu$ is trivial on  $\mathscr{A}(\mu)$.
If there is phase transition,
in the sense that the cardinality of $\mathscr{G}^{DLR}(f)$ is bigger than one,
then any extreme probability measure $\mu$
in $\mathscr{G}^{DLR}(f)$ is trivial on  $\mathscr{A}(\mu)$.

In the next proposition we show the relationship between
$\mathscr{A}(\mu)$ and $\mathscr{T}\equiv \bigcap_{n\in\mathbb{N}}\mathscr{T}_{n}$,
for $\mu\in \mathscr{G}^{DLR}$.

\begin{corollary} \label{ap5}
If $\mu\in\mathscr{G}^{DLR}(f)$
then $\mathscr{A}(\mu)$ is a  $\mu$ completion of
$\mathscr{T}$.
In particular, it follows from last corollary that if
$\mu\in\mathscr{G}^{DLR}(f)$ is extreme, then, it is trivial on
$\mathscr{T}$.
\end{corollary}

\begin{proof}
For all  $n\in\mathbb{N}$ we have that $K_n$ is a proper kernel.
Therefore, for any set
$F\in \mathscr{T}$
we get that
$K_n(F,\omega)=1_{F}(\omega)$.

On the other hand,
if
$F\in \{ E\in\mathscr{F}: K_n(E,\omega)
=1_{E}(\omega), \ \forall n\in\mathbb{N},\ \forall \omega\in\Omega \}$,
then,
$F=\{\omega\in\Omega: K_n(F,\omega)=1\}\in \mathscr{T}_{n}$.
Therefore,
$F\in\mathscr{T}$.
Consider $\mu\in\mathscr{G}^{DLR}(f)$ and let
$F\in \mathscr{A}(\mu)$, then,
\[
B = \bigcap_{n\in\mathbb{N}}\ \bigcup _{m\geq n} \{\omega\in\Omega: K_m(F,\omega)=1 \}
\]
is an element on the sigma algebra
$\mathscr{T}$ and moreover, $\mu(F\Delta B)=0$, because
\[
1_{B}
=
\limsup_{n\to\infty} 1_{ \{\omega\in\Omega: K_n(F,\omega)=1 \}  }
=
\limsup_{n\to\infty} 1_{F}
=
1_{F}
\qquad \mu\ a.e..
\qedhere\]
\end{proof}

\section*{Acknowledgments}
This study was financed in part by the Coordena\c  c\~ao de
Aperfei\c coamento de Pessoal de N\'ivel Superior - Brasil (CAPES) - Finance Code 001.
L. Cioletti and A. O. Lopes would like to acknowledge financial support by CNPq
and FAP-DF, and M. Stadlbauer acknowledges support by CNPq. The authors also would like to thank Evgeny Verbistkiy for pointing out an error in the proof
of Lemma 3 in the previous version of this manuscript and  Manfred Denker
for very helpful conversations during the process of writing this paper.

\end{document}